\documentclass[dvipdfm]{article}

\usepackage{amsmath, amsthm}
\usepackage{amsmath, amsfonts}
\usepackage{amsmath, amssymb}
\usepackage{amsmath}
\usepackage[all]{xy}

\newtheorem{thm}{Theorem}[section]
\newtheorem{cor}[thm]{Corollary}
\newtheorem{lem}[thm]{Lemma}
\newtheorem{prop}[thm]{Proposition}

\newtheorem{rmk}[thm]{Remark}

\numberwithin{equation}{section}

\newcommand{\prm}{\prime}

\newcommand{\pa}{\partial}

\newcommand{\bl}{\textbf}

 \makeatletter
 
 \newcommand{\Rmnum}[1]{\expandafter\@slowromancap\romannumeral #1@}
 \makeatother
\begin{document}

  \title{Picard-Fuchs Equations for Relative Periods and Abel-Jacobi Map for Calabi-Yau Hypersurfaces}
  \author{Si Li,\ Bong H. Lian,  \ Shing-Tung Yau}
  \date{}
  \maketitle


\begin{abstract}
We study the variation of relative cohomology for a pair consisting of a smooth projective hypersurface and an algebraic subvariety in it.  We construct an inhomogeneous Picard-Fuchs equation by applying a Picard-Fuchs operator to the holomorphic top form on a toric Calabi-Yau hypersurface, and deriving a general formula for the $d$-exact form on one side of the equation. We also derive a double residue formula, giving a purely algebraic way to compute the inhomogeneous Picard-Fuchs equations for Abel-Jacobi map, which has played an important role in recent study of D-branes \cite{D-brane Normal Function}.
Using the variation formalism, we prove that  the relative periods of toric B-branes on a toric Calabi-Yau hypersurface satisfy the enhanced GKZ-hypergeometric system proposed in physics literature \cite{Mayr}, and discuss the relations between the works \cite{D-brane Normal Function} \cite{Masoud} \cite{Mayr} in recent study of open string mirror symmetry. We also give the general solutions to the enhanced hypergeometric system.
\end{abstract}

\section{Introduction}
Mirror symmetry connects symplectic geometry of Calabi-Yau manifold
to complex geometry of its mirror manifold. In closed string theory, this has led to predictions on
counting curves on projective Calabi-Yau threefolds \cite{Candelas}\cite{BCOV}.
In open string theory, mirror symmetry has led to predictions on counting
holomorphic discs, first in the non-compact case studied in \cite{Mina
Vafa}\cite{Mina Vafa 2}, and more recently in the compact quintic example, where the
instanton sum of disc amplitude with non-trivial boundary on the
real locus of the real quintic is shown to be identical to the normalized
Abel-Jacobi map on the mirror quintic via mirror map \cite{open
mirror}\cite{D-brane Normal Function}\cite{solomon-walcher}.

In physics, the Abel-Jacobi map serves as the domain-wall tension of D-branes on the B-model,
and is obtained via reduction of the holomorphic Chern-Simons action on
curves \cite{Mina Vafa}. It is conjectured to have remarkable
integrality structure \cite{Ooguri-Vafa}. A key for
calculating the Abel-Jacobi map is through inhomogeneous
Picard-Fuchs equations \cite{D-brane Normal Function}. Let $X_z$ be a family of Calabi-Yau threefolds parameterized by variable $z$, and $\Omega_z$ be a family of nonzero holomorphic 3-forms on $X_z$. Assume that there is a family of pairs of holomorphic curves $C^+_z, C^-_z$ in $X_z$. Let $\mathcal D(\pa_z)$ be a Picard-Fuchs
operator. Then there exists a 2-form $\beta_z$ such that
\begin{eqnarray}
    \mathcal D(\pa_z)\Omega_z=-d\beta_z
\end{eqnarray}
The exact term $d\beta_z$ does not contribute when it is integrated over a closed 3-cycle $\Gamma$ in $X_z$. The so-called closed-string period $\int_\Gamma\Omega_z$ then satisfies a homogeneous Picard-Fuchs equation.
In open string theory, it is necessary to consider the integral of $\Omega_z$ over a 3-chain $\Gamma$ in $X_z$ which is not closed, but whose boundary is $C^+-C^-$.
Because of contributions from the boundary, this so-called open-string period $\int_\Gamma\Omega_z$ satisfies an inhomogeneous Picard-Fuchs equation. Solving the equation gives a precise description of the Abel-Jacobi map up to closed-string periods. To study this map, $\beta_z$ plays an essential role since it is this form that gives rise to one side of the inhomogeneous Picard-Fuchs equation:
\begin{eqnarray}
    \int_\Gamma \mathcal D(\pa_z) \Omega_z=-\int_{\partial\Gamma}\beta_z.
\end{eqnarray}
The inhomogeneous term on the right side turns out also to encodes important information for predicting the number of holomorphic disks on a mirror Calabi-Yau manifold.

There have been several proposals for constructing the inhomogeneous Picard-Fuchs equation and its solutions. In the case of 1-moduli family \cite{D-brane Normal
Function}, it was done by first computing $\beta_z$ using the Griffith-Dwork reduction
procedure, and then by doing an explicit (but delicate) local analytic calculation of appropriate
boundary integrals. Based on the notion of  off-shell mirror symmetry, two other proposals \cite{Masoud}\cite{Mayr} have been put forth. Roughly speaking, their setup begins with a family of divisors $Y_{z,u}$ which deforms in $X_z$ under an additional parameter $u$. For each relative homology class $ \Gamma\in H_3(X_z, Y_{z,u})$, one considers the integral
\begin{eqnarray}
    \int_{\Gamma}\Omega_z
\end{eqnarray}
which is called a relative period for B-brane. It is proposed that
the open-string periods above be recovered as a certain critical value of the relative period, regarded as a function of $u$. To calculate the relative periods, \cite{Masoud} proposed a procedure similar to the Griffith-Dwork reduction. In \cite{Mayr}, an enlarged polytope is  proposed to encode both the
geometry of the Calabi-Yau $X_z$ and the B-brane geometry. This gives rise to a GKZ hypergeometric system for the relative periods, and a special solution at a critical point in $u$ then leads to a solution to the original inhomogeneous Picard-Fuchs equation.

Our goal in this paper is to further develop the mathematical structures underlying inhomogeneous Picard-Fuchs equations and the Abel-Jacobi map, and to clarify the relationships between the three approaches mentioned above. Here is an outline.
We begin, in section 2, with a description of a residue formalism for relative cohomology of a family of pairs $(X_z,Y_z)$, including a number of variational formulas on the local system $H^n(X_z,Y_z)$. In section 3, we derive a general formula for the exact form (the $\beta$-term) appearing in the inhomogeneous Picard-Fuchs equation for toric Calabi-Yau hypersurfaces, generalizing GKZ-type differential equation to the level of differential forms instead of cohomology classes. This gives a much more uniform approach to computing the $\beta$-term than the Griffith-Dwork reduction. 
In section 4, we prove a purely algebraic a double residue formula for the inhomogeneous term of the Picard-Fuchs equation that governs the Abel-Jacobi map. This uniform approach also allows us to bypass the delicate local analytical calculation of boundary integrals in a previous approach \cite{D-brane Normal Function}\cite{one-par open mirror}. 
In section 5, using the residue formalism in section 2, we give a simple interpretation of the relative version of the Griffith-Dwork
reduction used in \cite{Masoud}. In particular, this gives a mathematical justification for the appearance of log divisor, and elucidates the relationship between relative periods and the Abel-Jacobi map. 
We also give a uniform description for the enhanced polytope method for describing toric B-brane
geometry in a general toric Calabi-Yau hypersurface, and show that relative periods satisfy the corresponding enhanced GKZ system.  Finally, we give a general formula, modeled on the closed string case \cite{HKSY}\cite{Hosono-Lian}, for solution to the enhanced GKZ system.

{\it Acknowledgement.} S.L. would like to thank J. Walcher for many stimulating discussions, and thank M.Soroush for answering many questions on his paper. After the completion of a preliminary draft of our paper, three other papers \cite{AB}\cite{AHJMMS}\cite{GHKK} with some overlap with ours have since been posted on the arXiv.

\section{Variation of Relative Cohomology}
\subsection*{Local System of Relative Cohomology and Gauss-Manin Connection}
Let $\pi: \mathcal X\to S$ be a smooth family of $n$-dimensional
projective varieties, and $\mathcal Y\to S$ be a family of smooth
subvariety $\mathcal Y\subset \mathcal X$. Let $s\in S$ be a closed point, and denote by $X_s, Y_s$
the corresponding fiber over $s$. Consider the family of relative cohomology
class
$$
    H^n(X_s,Y_s)
$$
given by the cohomology of the complex of pairs:
$$
    \Gamma(\Omega^n(X_s))\oplus \Gamma(\Omega^{n-1}(Y_s))
$$
with the differential
\begin{eqnarray}
    d(\alpha,\beta)=(d\alpha, \alpha|_{Y_s}-d\beta)
\end{eqnarray}
Here $\Omega^n(X_s)$ and $\Omega^{n-1}(Y_s)$ are sheaves of De Rham
differential $n$-forms on $X_s$ and $(n-1)$-form on $Y_s$, and $\Gamma$ is the smooth global section.
 Therefore an element of $H^n(X_s, Y_s)$ is represented by a
differential $n$-form on $X_s$ whose restriction to $Y_s$ is specified by an exact
form.

\begin{lem}
$H^n(X_s,Y_s)$ forms a local system on $S$.
\end{lem}
\begin{proof}
The proof is similar to the case without $\mathcal Y$ by choosing a
local trivialization of $\mathcal X\to S$ which also trivializes
$\mathcal Y\to S$. See e.g.\cite{Voisin}.
\end{proof}

We denote this local system by $\mathcal H^n_{(\mathcal X, \mathcal Y)}$, and let $\nabla^{GM}$ be the Gauss-Manin connection. \\
There's a well-defined natural pairing
\begin{eqnarray}
\begin{array}{ccccc}
    H_n(X_s, Y_s) &\otimes & H^n(X_s, Y_s) &\to& \mathbb C\\
    \Gamma & \otimes & (\alpha, \beta) &\mapsto& <\Gamma, (\alpha, \beta)>\equiv \int_\Gamma \alpha- \int_{\pa \Gamma} \beta.
\end{array}
\end{eqnarray}
\\
Given a family $(\alpha_s, \beta_s)\in H^n(X_s, Y_s)$ varying
smoothly, which gives a smooth section of $\mathcal H^n_{(\mathcal X,
\mathcal Y)}$ denoted by $[(\alpha_s, \beta_s)]$, and $\Gamma_s\in
H_n(X_s, Y_s)$  a smooth family of relative cycles, we get a function on
$S$ given by the pairing
$$
    <\Gamma_s, (\alpha_s, \beta_s)>=\int_{\Gamma_s} \alpha_s- \int_{\pa \Gamma_s} \beta_s
$$
Let $v$ be a vector field on $S$. We consider the variation
$$
    \mathcal L_v <\Gamma_s, (\alpha_s, \beta_s)>
$$
where $\mathcal L_v$ is the Lie derivative with respect to $v$. Suppose we have a lifting $\tilde \alpha, \tilde \beta$, which are differential forms on $\mathcal X, \mathcal Y$ respectively, such that
$$
    \tilde \alpha|_{X_s}=\alpha_s, \ \tilde \beta |_{Y_s}=\beta_s
$$
and that $\Gamma_s$ moves smoothly to form a cycle $\tilde \Gamma$ on $\mathcal X$:
$$
    \Gamma_s=\tilde \Gamma \cap X_s, \ \ \pa \Gamma\subset \mathcal Y.
$$
Let $\tilde v_{\mathcal X}$ be a lifting of $v$ on $\mathcal X$,
$\tilde v_{\mathcal Y}$ be a lifting of $v$ on $\mathcal Y$.

\begin{prop}[Variation Formula]
    \begin{eqnarray}\label{variation formula}
        \mathcal L_v <\Gamma_s, (\alpha_s, \beta_s)>=<\Gamma_s, (\iota_{\tilde v_{\mathcal X}}\lrcorner d\tilde \alpha, \iota_{\tilde v_{\mathcal Y}}
        \lrcorner(d\tilde \beta-\tilde \alpha) )>
    \end{eqnarray}
    where $\iota_{\tilde v_{\mathcal X}}$ is the contraction with $\tilde v_{\mathcal
    X}$, and similarly for $\iota_{\tilde v_{\mathcal Y}}$.
\end{prop}
\begin{proof}
    We fix a point $s_0\in S$, and let $\sigma(t)$ be a local
    integral curve of $v$ such that $\sigma(0)=s_0$. Let
    $\tilde \Gamma_{\sigma(t)}$ denote the one-dimensional family of cycles over
    $\sigma(s), \ \ 0\leq s\leq t$, and we denote by $\Gamma_{t}$ the cycle over
    the point $\sigma(t)$. Also let $\widetilde{\pa\Gamma}_{\sigma(t)}$ be the
    family of boundary cycle $\pa\Gamma_s$ over $\sigma(s), 0\leq s\leq t$. Then we
    have
    \begin{eqnarray}
    \pa \left(\tilde
    \Gamma_{\sigma(t)}\right)=\Gamma_t-\Gamma_0-\widetilde{\pa\Gamma}_{\sigma(t)}
    \end{eqnarray}
    therefore
    \begin{eqnarray*}
        \int_{\tilde\Gamma_{\sigma(t)}} d\tilde \alpha=\int_{\Gamma_t}\tilde
    \alpha-\int_{\Gamma_0}\tilde
    \alpha-\int_{\widetilde{\pa\Gamma}_{\sigma(t)}}\tilde\alpha
    \end{eqnarray*}
    Similarly
    \begin{eqnarray*}
        \int_{\widetilde{\pa\Gamma}} d\tilde\beta = \int_{\pa \Gamma_t}\tilde\beta-\int_{\pa\Gamma_0}\tilde\beta
    \end{eqnarray*}
    Taking the derivative with respect to t, we get
    \begin{eqnarray*}
        {\pa\over \pa t} \left(\int_{\Gamma_t}\tilde \alpha -\int_{\pa \Gamma_t}\tilde\beta \right)=\int_{\Gamma_t}\iota_{\tilde v_{\mathcal X}}\lrcorner d\tilde \alpha
        +\int_{\pa\Gamma_t}\iota_{\tilde v_{\mathcal Y}}\lrcorner \left(\tilde \alpha-d\tilde\beta\right)
    \end{eqnarray*}
The proposition follows.
\end{proof}

Note that $(\iota_{\tilde v_{\mathcal X}}\lrcorner d\tilde \alpha,
\iota_{\tilde v_{\mathcal Y}}
        \lrcorner(d\tilde \beta-\tilde \alpha) )$ is nothing but the
        Gauss-Manin connection
 \begin{eqnarray}\label{Gauss Manin}
        \nabla^{GM}_v [(\alpha_s, \beta_s)]=\left[\left((\iota_{\tilde v_{\mathcal X}} \lrcorner d\tilde \alpha )|_{X_s}, (\iota_{\tilde v_{\mathcal Y}}\lrcorner
        (d\tilde \beta-\tilde \alpha) )|_{Y_s}\right)\right]
 \end{eqnarray}
and it is straightforward to check using the variation formula that the right side of (\ref{Gauss Manin}) is
independent of the choice of $\tilde \alpha, \tilde \beta, \tilde
v_{\mathcal X}, \tilde v_{\mathcal Y}$, and that the connection is flat. The following corollary also follows from (\ref{Gauss Manin}).

\begin{cor}
\begin{eqnarray}
 \mathcal L_v <\Gamma_s, (\alpha_s, \beta_s)>=<\Gamma_s, \nabla^{GM}_v [(\alpha_s, \beta_s)]>
\end{eqnarray}
\end{cor}

\subsection*{Residue Formalism for Relative Cohomology}
In this section, we assume that $X_z$ moves as a family of
hypersurfaces in a fixed $n+1$-dim ambient projective space $M$ with defining
equation $P_z=0$. Here $P_z\in H^0(M,[D])$ for a fixed divisor class
$[D]$, and $z$ is holomorphic coordinate on $S$ parametrizing the family. Let
$$
    \omega_z\in H^0(M, K_M(X_z))
$$
be a rational (n+1,0)-form on $M$ with pole of order one along
$X_z$, then we get a famliy of holomorphic (n,0)-form on $X_z$ given
by
$$
    Res_{X_z} \omega_z \in H^{0}(X_z, K_{X_z})
$$
and also a family of relative cohomology classes
$$
    (Res_{X_z} \omega_z,0) \in H^n(X_z, Y_z)
$$
Here our convention for $Res$ is that if $X_z$ is locally given by $w=0$, and
$\omega_z={dw\over w}\wedge \phi$, where $\phi$ is locally a smooth form, then  $Res_{X_z}\omega_z=\phi|_{X_z}$.
Note that the map $Res_{X_z}: H^0(M, K_M(X_z))\to H^{n,0}(X_z)$ is
at the level of forms, not only as cohomology classes since the
order of pole is one. While in the residue formalism of ordinary cohomology, we can ignore
exact forms to reduce the order of the pole \cite{Rational Integral}, it is important to keep track of the order of the
pole in considering periods of relative cohomology because of the boundary term.  \\

We choose a fixed open cover $\{U_\alpha\}$ of $M$ and a partition of unity $\{\rho_\alpha\}$ subordinate to it. Let $P_{z,\alpha}=0$ be the defining equation of $X_s$ on $U_\alpha$. Then we can write
\begin{eqnarray}
    \omega_z=\sum_{\alpha} {d_M P_{z, \alpha}\over P_{z, \alpha}}\wedge \phi_{z,\alpha}
\end{eqnarray}
where $\phi_{z,\alpha}$ is a smooth $(n,0)$-form with $ \overline{\ supp(\phi_{z,\alpha})}\subset U_\alpha$.
On the trivial family $M\times S$, we will use $d_M$ to denote the
differential along $M$ only and use $d$ to denote the differential
on the total space. Let
\begin{eqnarray}
    \phi_z=\sum_\alpha \phi_{z, \alpha}
\end{eqnarray}
Then $\phi_z$ is a smooth form on $M\times S$ such that
$$
    \phi_z|_{X_z}=Res_{X_z}  {\omega_z}
$$
Consider the variation
\begin{eqnarray*}
    {\pa\over \pa z} \omega_z    &=& d_M\left( \sum_\alpha  \pa_z \log(P_{z,\alpha})  \phi_{z,\alpha}\right)+\sum_{\alpha} \left({d_M P_{z, \alpha}\over P_{z, \alpha}}\wedge \pa_z\phi_{z,\alpha}-{\pa_z P_{z,\alpha} \over P_{z,\alpha}}d_M \phi_{z,\alpha}\right)
\end{eqnarray*}
Let
$$
    \tilde{\pa_{z}}_{\mathcal X}={\pa\over \pa z}+ n_{z,\mathcal X}
$$
be a lifting of ${\pa\over \pa z}$ to $\mathcal X$, where $n_{z,\mathcal X}\in \Gamma(T_M|_{X_z})$ is along the fiber, which is a normal vector field corresponding to the deformation of $X_z$ in $M$ with respect to $z$. Then in each
$U_\alpha$, we have
\begin{eqnarray}
    \left(\iota_{n_{X_z}}\lrcorner d_M P_{z,\alpha}\right)|_{X_z}=-{\pa_z} P_{z,
    \alpha}|_{X_z}
\end{eqnarray}
It follows easily that
\begin{eqnarray}
    \iota_{\tilde{\pa_{z}}_{\mathcal X}}\lrcorner d \phi_z |_{X_z}=Res_{X_z} \left( \pa_z\omega_z - d_M \left( \pa_z \log(P_{z})  \phi_{z}\right) \right)
\end{eqnarray}
Note that since the transition function of $[D]$ is independent of $z$, $\pa_z
\log(P_{z})$ is globally well-defined. In general, $\pa_z \omega_z$ will
have a pole of order two along $X_z$, but the substraction
of $d_M \left( \pa_z \log(P_{z})  \phi_{z}\right)$ makes it logarithmic along $X_z$, hence the residue above is
well-defined.\\

Next we choose arbitrary lifting of ${\pa\over \pa z}$ to $\mathcal
Y$, and write it as
$$
    \tilde \pa_{z,\mathcal Y}={\pa\over \pa z}+n_{z,\mathcal Y}
$$
where $n_{z,\mathcal X}\in \Gamma(T_M|_{Y_z})$ is along the fiber, which is a normal vector field corresponding to the deformation of $Y_z$ in $M$ with respect to $z$. Then the variation formula implies that

\begin{prop}[Residue Variation Formula]
\begin{eqnarray}\label{variation res}
    \nabla^{GM}_{\pa_z} \left( Res_{X_z} {\omega_z},0 \right)=\left(Res_{X_z} \left(\pa_z{\omega_z} - d_M \left( \pa_z \log(P_{z})  \phi_{z}\right) \right), - \iota_{n_{z,\mathcal Y}}\lrcorner \phi_z\right)
\end{eqnarray}
\end{prop}

To see the effect of the second component on the right side, let us assume that $Y_z=X_z\cap
H$, where $H$ is a fixed hypersurface in $M$ with defining equation
$Q=0$. Then we can choose a $\epsilon$-tube $T_\epsilon(\Gamma)$
\cite{Rational Integral} of $\Gamma\in H_n(X_z, Y_z)$ with $\pa
T_{\epsilon}(\Gamma)\subset H$. Hence
\begin{eqnarray*}
    {1\over 2\pi i}\int_{T_\epsilon(\Gamma)} - d_M \left( \pa_z \log(P_{z})  \phi_{z}\right)&=&{1\over 2\pi i}\int_{T_\epsilon(\pa\Gamma)}\left(\pa_z \log(P_{z})  \phi_{z}
    \right)|_{H}\\
    &=&\int_{\pa \Gamma} Res_{Y_z}\left(\pa_z \log(P_{z})  \phi_{z}
    \right)|_{H}\\
    &=&-\int_{\pa \Gamma} \iota_{n_{z,\mathcal Y}}\lrcorner \phi_z
\end{eqnarray*}
which cancels exactly the second component on the right side of (\ref{variation res}). Therefore,
$$
    \pa_z <\Gamma,  \left( Res_{X_z} {\omega_z},0 \right)>={1\over 2\pi i}\int_{T_\epsilon(\Gamma)} \pa_z \omega_z
$$

We can localize the above observation and consider the following
situation: on each $U_\alpha$, suppose we can choose $Q_{\alpha}$
independent of $z$ such that $P_{z,\alpha}, Q_{\alpha}$ are
transversal and $Y_{z}\cap U_\alpha \subset \{Q_{\alpha}=0,
P_{z,\alpha}=0 \}$. Suppose we have a relative cycle $\Gamma\in
H_n(X_z, Y_z)$ where we can choose a $\epsilon$-tube
$T_\epsilon(\Gamma)$ such that $\pa T_\epsilon(\Gamma)\cap U_\alpha$
lies in $\{Q_{\alpha}=0\}$. We have the pairing
\begin{eqnarray*}
    <\Gamma, \left(Res_{X_z}\omega_z, 0\right)>=\lim_{\epsilon\to 0}{1\over 2\pi i}\int_{T_{\epsilon}(\Gamma)}\omega_z
\end{eqnarray*}
However, the right hand side doesn't depend on $\epsilon$. In fact, let $T^\epsilon_\delta(\Gamma)$ be a solid annulus over $\Gamma$. By Stokes's theorem, we have
$$
    \int_{T_{\epsilon}(\Gamma)}\omega_z-\int_{T_{\sigma}(\Gamma)}\omega_z=\int_{\pa T^\epsilon_\delta(\Gamma)}\omega_z
$$
since $\pa T^\epsilon_\delta(\Gamma)\cap U_\alpha\subset
\{Q_\alpha=0\}$ for each $\alpha$, the above integral vanishes.
Therefore the integral $\int_{T_{\epsilon}(\Gamma)}\omega_z$ doesn't
depend on the position of the $\epsilon$-tube if we impose the
boundary condition as above. It follows immediately that
$$
    ({\pa_z})^k <\Gamma, \left(Res_{X_z}\omega_z, 0\right)>={1\over 2\pi i}\int_{T_{\epsilon}(\Gamma)}({\pa_z})^k \omega_z
$$
where $\pa T_\epsilon(\Gamma)\cap U_\alpha$ lies in $\{Q_{\alpha}=0\}$. Applying this to a Picard-Fuchs operator, we get

\begin{prop}\label{Inhomog}[Inhomogeneous Picard-Fuchs Equation]
Let $\mathcal D=\mathcal D({\nabla_{\pa_z}^{GM}})$ be a Picard-Fuchs operator, i.e.
\begin{eqnarray}
    \mathcal D(\pa_z) \omega_z=-d\beta_z
\end{eqnarray}
for some rational $(n-1,0)$-form $\beta_z$ with poles along $X_z$.
Then under the above local choices, we have
\begin{eqnarray}\label{localization}
 \mathcal D(\pa_z) <\Gamma, \left(Res_{X_z}\omega_z, 0\right)>={1\over 2\pi i}\int_{T_\epsilon(\pa\Gamma)}\beta_z.
\end{eqnarray}
\end{prop}
In the next section, we will derive a general formula for $\beta_z$ using toric method.

 More generally, suppose that $\{Q_\alpha\}$ depends on $z$
which is denoted by $Q_{z,\alpha}$, and we put our $\epsilon$-tube
$T_\epsilon(\Gamma_z)$ inside $\{Q_{z,\alpha}\}=0$ on $U_\alpha$.
Then
\begin{eqnarray*}
    \int_{\Gamma_z} Res_{X_z} \omega_z={1\over 2\pi i}\int_{T_{\epsilon}(\Gamma_z)} \omega_z=\sum_{\alpha}{1\over 2\pi i}
    \int_{T_{\epsilon}(\Gamma_z)} \rho_\alpha \omega_z
\end{eqnarray*}
and the integration doesn't depend on $\epsilon$ assuming the boundary condition as above. Applying the variation formula (\ref{variation formula}) we get
\begin{eqnarray*}
    {\pa\over \pa z}\sum_{\alpha}{1\over 2\pi i}
    \int_{T_{\epsilon}(\Gamma_z)} \rho_\alpha \omega_z=\sum_{\alpha}{1\over 2\pi i}
    \int_{T_{\epsilon}(\Gamma_z)} \rho_\alpha \pa_z\omega_z-\sum_{\alpha}{1\over 2\pi i}
    \int_{T_{\epsilon}(\pa\Gamma_z)} \rho_\alpha \left(\iota_{n_{z, Q_\alpha}}\lrcorner\omega_z\right)|_{Q_{z,\alpha}=0}
\end{eqnarray*}
where $n_{z,Q_\alpha}$ is the normal vector field corresponding to the deformation of $\{Q_{z,\alpha}=0\}$ inside $M$. In particular, we have
$\iota_{n_{z, Q_\alpha}}\lrcorner dQ_{z,\alpha}|_{\{Q_{z,\alpha=0}\}}=-\pa_z Q_{z,\alpha}|_{\{Q_{z,\alpha=0}\}}$. Hence
$$
    \left(\iota_{n_{z, Q_\alpha}}\lrcorner\omega_z\right)|_{Q_{z,\alpha}=0}=-Res_{Q_{z,\alpha}=0}\left(\pa_z\log (Q_{z,\alpha}) \omega_z\right)
$$
Putting together the last three equations, we arrive at 

\begin{prop}[cf. \cite{mirror calculation}]
\begin{eqnarray*}
    {\pa^k\over \pa z^k}\left( \int_{\Gamma_z} Res_{X_z} \omega_z \right)={1\over 2\pi i}\int_{T_{\epsilon}(\Gamma_z)} \pa_{z}^k \omega_z
    +\sum_{l=1}^k {\pa^{k-l}\over \pa z^{k-l}} \sum_{\alpha}{1\over 2\pi i}
    \int_{T_{\epsilon}(\pa\Gamma_z)} \rho_\alpha Res_{Q_{z,\alpha}=0}\left(\pa_z\log (Q_{z,\alpha}){\pa^{l-1}\over \pa
    z^{l-1}} \omega_z\right)
\end{eqnarray*}
\end{prop}


\section{Exact GKZ Differential Equation and Toric geometry}

In this section, we study the Picard-Fuchs differential operators arising
from a generalized GKZ hypergeometric systems
\cite{HKSY}\cite{Hosono-Lian} for toric Calabi-Yau hypersurfaces and
derive a general formula for the $\beta$-term of an inhomogeneous Picard-Fuchs equation, from toric data.\\

We first consider the special case of a weighted projective space, where $\beta$-term will be much simpler than in the general case, which will be considered at the end of this section. Let
$\mathbb{P}^4(\bl{w})=\mathbb{P}^4(w_1,w_2,w_3,w_4,w_5)$. We assume
that $w_5=1$ and it's of Fermat-type, i.e., $w_i|d$ for each i,
where $d=w_1+w_2+w_3+w_4+w_5$. There's associated 4-dimensional
integral convex polyhedron given by the convex hull of the integral
vectors
\begin{eqnarray*}
    \Delta=\left\{(x_1,\cdots,x_5)\in \mathbb{R}^5|\sum_{i=1}^5 w_i x_i=0, x_i\geq -1   \right\}
\end{eqnarray*}
If we choose the basis $\{e_i=(1,0,0,0,-w_i), i=1..4\}$, then the
vertices is given by
\begin{eqnarray*}
                \Delta: && v_1=({d\over w_1}-1,-1,-1,-1) \\&& v_2=(-1,{d\over w_2}-1,-1,-1) \\&&
             v_3=(-1,-1,{d\over w_3}-1,-1)\\ && v_4=(-1,-1,-1,{d\over w_4}-1)\\ && v_5=(-1,-1,-1,-1)
\end{eqnarray*}
and the vertices of its dual polytope is given by
\begin{eqnarray*}
            \Delta^*: && v^*_1=(1,0,0,0)\\ && v^*_2=(0,1,0,0)\\ &&
              v^*_3=(0,0,1,0)\\ && v^*_4=(0,0,0,1)\\ && v^*_5=(-w_1,-w_2,-w_3,-w_4)
\end{eqnarray*}
Let $v_k^*, k=0,1,2,..,$ be integral points of $\Delta^*$, where
$v_0^*=(0,0,0,0)$. We write
\begin{eqnarray}
    f_{\Delta^*}(x)=\sum_{v_k^*\in \Delta^*}a_k X^{v_k^*}
\end{eqnarray}
which is the defining equation for our Calabi-Yau hypersurfaces in
the anti-canonical divisor class. Here $X=\{X_1,X_2,X_3,X_4\}$ is the toric coordinate, $ X^{v_k^*}=\prod\limits_{j=1}^4 X_j^{v^*_{k,j}}$. If we use homogeneous coordinate
$\cite{Homogeneous Coord}$ $\{z_\rho, 1\leq \rho \leq 5\}$
corresponding to the one-dim cone $\{v_\rho, 1\leq \rho\leq 5\}$,
then the toric coordinate can be written by homogeneous coordinate
\begin{eqnarray}
    X_j={z_j^{d/w_j}\over \prod\limits_{\rho=1}^5 z_\rho}, \ \ \ j=1,2,3,4
\end{eqnarray}
The relevant rational form with pole of order one along the
hypersurface is given by
\begin{eqnarray}
    \Pi(a)&=&{\prod\limits_{i=1}^5 w_i\over d^3}{1\over  \sum_{v_k^*\in \Delta^*}a_k X^{v_k^*} } \prod_{j=1}^4 {dX_j\over X_j} \nonumber\\
    &=& {\Omega_0\over \sum\limits_{v_k^*\in \Delta^*} a_k \prod\limits_{\rho=1}^5 z_\rho^{<v_k^*,v_\rho>+1}}\nonumber\\
    &=&{\Omega_0\over  a_0 \prod\limits_{\rho=1}^5 z_\rho+\sum\limits_{\rho=1}^5 a_\rho z_\rho^{d/w_\rho}+ \sum\limits_{v_k^*\in \Delta^*, k>5} a_k \prod\limits_{\rho=1}^5 z_\rho^{<v_k^*,v_\rho>+1}}
\end{eqnarray}
where $\Omega_0=\sum_{\rho=1}^5 (-1)^{\rho-1}w_\rho z_\rho
dz_1\wedge \cdots \hat{dz_\rho}\wedge\cdots dz_5$.  Define the
relation lattice by
$$
    L=\{ l=(l_0,l_1,...)\in \mathbb Z^{|\Delta^*|+1} | \sum_i l_i \bar v_i^*=0 \}, \ \ \mbox{where}\ \ \bar v_i^*=(1,v_i^*), \ v_i^*\in \Delta^*
$$
The moduli variable associated with the choice of a basis
$\{l^{(k)}\}$ for $L$ is given by \cite{HKSY}
$$
    x_k=(-1)^{l_0^{(k)}}a^{l^{(k)}}
$$
The key idea here is to consider the following 1-parameter family of automorphisms
$$
    \phi_t: z_\rho\to ({a_0\over a_\rho})^{{w_\rho\over d}t}z_\rho, \ \ \ 1\leq \rho \leq 5.
$$
Put
$$
    \Pi_t(a)=\phi_t^* \Pi(a).
$$
It satisfies the differential equation
$$
    {\pa_t} \left(\phi_t^* \Pi(a)\right)=\mathcal L_{V}\phi_t^*\Pi(a)
$$
where $V=\sum\limits_{\rho=1}^5{w_\rho\over d}(\log{a_0\over a_\rho}) z_\rho {\pa\over \pa z_\rho}$ is the generating vector field for $\phi_t$,
$\mathcal L_V$ is the Lie derivative. This is solved by
\begin{eqnarray}
\phi_t^* \Pi(a)=e^{t\mathcal L_V}\Pi(a)
\end{eqnarray}
Define
\begin{eqnarray}
    \tilde \Pi(x)=a_0 \Pi_1(a)
\end{eqnarray}
$\tilde \Pi(x)$ is a function of $\{x_k\}$
only.  Indeed,
\begin{eqnarray*}
    \tilde \Pi(x)={  \Omega_0\over  \prod\limits_{\rho=1}^5 z_\rho+ \left(\prod\limits_{\rho=1}^5({a_\rho\over a_0})^{w_\rho/d} \right)\sum\limits_{\rho=1}^5  z_\rho^{d/w_\rho}+\sum\limits_{v_k^*\in \Delta^*, k>5} {a_k\over a_0} \prod\limits_{\rho=1}^5\left( {a_0\over a_\rho}\right)^{{w_\rho\over d}<v_k^*,v_\rho>} \prod\limits_{\rho=1}^5 z_\rho^{<v_k^*,v_\rho>+1}}
\end{eqnarray*}
Since we have
\begin{eqnarray*}
    \sum_{\rho=1}^5 w_\rho v_\rho=0, \ \ v_k^*=\sum_{\rho=1}^5 {w_\rho\over d}<v_k^*,v_\rho>v_\rho^*
\end{eqnarray*}
we see that both $\left(\prod\limits_{\rho=1}^5({a_\rho\over a_0})^{w_\rho/d} \right)$ and ${a_k\over a_0} \prod\limits_{\rho=1}^5\left( {a_0\over a_\rho}\right)^{{w_\rho\over d}<v_k^*,v_\rho>} $ can be written in terms of $x_k$'s as an algebraic function.

Given an integral point $l\in L$, consider the GKZ operator (it differs from the standard GKZ operator by a factor of $a_0\prod_{l_i>0}a_i^{l_i}$)
\begin{eqnarray*}
    D_l&=& a_0\left\{ \prod_{l_i>0}a_i^{l_i}\left({\pa\over \pa a_i}\right)^{l_i} -a^l  \prod_{l_i<0}a_i^{-l_i}\left({\pa\over \pa a_i}\right)^{-l_i}  \right\}\\
    &=&\left\{\prod_{j=1}^{l_0}(a_0{\pa\over \pa a_0}-j) \prod_{i\neq 0,l_i>0} \prod_{j=0}^{l_i-1}(a_i{\pa\over \pa a_i}-j)-a^l \prod_{j=1}^{-l_0}(a_0{\pa\over \pa a_0}-j) \prod_{i\neq 0, l_i<0} \prod_{j=0}^{-l_i-1} (a_i{\pa\over \pa a_i}-j) \right\}a_0\\
    &=& \tilde D_l a_0
\end{eqnarray*}
where we use the convention that $\prod\limits_{i=1}^m(\cdots)=1$ if $m\leq 0$. From
$$
    D_l \Pi(a)=0
$$
We get
\begin{eqnarray}
    e^{\mathcal L_V}\tilde D_l e^{-\mathcal L_V}\tilde \Pi(x)=0
\end{eqnarray}

\begin{lem}
    $$
        e^{\mathcal L_V} (a_i{\pa\over \pa a_i})e^{-\mathcal L_V}=a_i{\pa\over \pa a_i}+\delta_{1\leq i\leq 5}\mathcal L_{{w_i\over d}z_i{\pa\over \pa z_i}}-\delta_{i,0}\mathcal L_{\sum\limits_{\rho=1}^5 {w_\rho\over d}z_\rho{\pa\over \pa z_\rho}}
    $$
    here $\delta_{1\leq i\leq 5}=1$ if $1\leq i\leq 5$ and otherwise $0$.
\end{lem}
\begin{proof}Since
    \begin{eqnarray*}
        \left[V, a_i{\pa\over a_i}\right]&=&\delta_{1\leq i\leq 5}{{w_i\over d}z_i{\pa\over \pa z_i}}-\delta_{i,0}{\sum\limits_{\rho=1}^5 {w_\rho\over d}z_\rho{\pa\over \pa z_\rho}}\\
        \left[V,\left[V,a_i{\pa\over a_i}\right]\right]&=&0
\end{eqnarray*}
The lemma follows from the formula
\begin{eqnarray*}
     e^{\mathcal L_V} (a_i{\pa\over \pa a_i})e^{-\mathcal L_V}=a_i{\pa\over \pa a_i}+\sum_{k=1}^\infty {1\over k!}\mathcal L_{(ad_V)^k \left(a_i{\pa\over \pa a_i}\right)}
\end{eqnarray*}
\end{proof}
It follows from the lemma that
\begin{eqnarray}\label{PF-op}
    \left\{\prod_{j=1}^{l_0}(a_0{\pa\over \pa a_0}-j)\right. & \prod\limits_{i\neq 0, l_i>0}& \prod_{j=0}^{l_i-1} (a_i{\pa\over \pa a_i}-j+\delta_{1\leq i\leq 5}\mathcal L_{{w_i\over d}z_i{\pa\over \pa z_i}}) \nonumber\\
  -a^l \prod_{j=1}^{-l_0}(a_0{\pa\over \pa a_0}-j) &\prod\limits_{i\neq 0, l_i<0} &  \left.\prod_{j=0}^{-l_i-1} (a_i{\pa\over \pa a_i}-j+\delta_{1\leq i\leq 5}\mathcal L_{{w_i\over d}z_i{\pa\over \pa z_i}}) \right\}\tilde\Pi(x)=0
\end{eqnarray}
where we have used $\mathcal L_{\sum\limits_{\rho=1}^5 {w_\rho\over d}z_\rho{\pa\over \pa z_\rho}}\tilde \Pi(x)=0$ to eliminate the terms with $\mathcal L_{\sum\limits_{\rho=1}^5 {w_\rho\over d}z_\rho{\pa\over \pa z_\rho}}$. 

Observe that each Lie derivative $\mathcal L_{{w_i\over d}z_i{\pa\over \pa z_i}}$ commutes with all other operators appearing on the left side of (\ref{PF-op}). So we can move every term involving $\mathcal L_{{w_i\over d}z_i{\pa\over \pa z_i}}$ to the right side, so that (\ref{PF-op}) can now be explicitly written as
$$
    \tilde D_l \tilde \Pi(x)=-\sum_{i=1}^5 \mathcal L_{{w_i\over d}z_i{\pa\over \pa z_i}}\alpha_i
$$
where the $\alpha_i$ are (easily computable) $d$-closed 4-forms depending on $l$.  By the Cartan-Lie formula $\mathcal L_X=d\iota_X+\iota_X d$, we obtain the formula

\begin{prop}\label{beta term}[$\beta$-term Formula]
$$
    \tilde D_l \tilde \Pi(x)=-d\beta_l
$$
where $\beta_l=\sum_i \iota_{{w_i\over d}z_i{\pa\over \pa z_i}}\alpha_i$.
\end{prop}

Next we consider the differential operators of the extended GKZ system induced by the automorphism of the ambient toric variety. The corresponds to the root of $\Delta^*$ \cite{Hosono-Lian}. Let
$$
    v_i^*\in R(\Delta^*), \ \ <v_i^*,v_{\rho_i}>=-1, \ \ <v_i^*, v_\rho>\geq 0 \ \mbox{for}\ \rho\neq \rho_i
$$
then we obtain an equation
\begin{eqnarray*}
    \left\{ \sum_{v_k^*\in \Delta^*}(<v_k^*, v_{\rho_i}>+1)a_{v_k^*}{\pa\over \pa a_{v_k^*+v_i^*}} \right\}{1\over a_0}e^{-\mathcal L_V}\tilde\Pi(x)
    =\mathcal L_{\left\{\prod\limits_{\rho=1}^5z_\rho^{<v_i^*, v_\rho>}\right\} z_{\rho_i}{\pa\over \pa z_{\rho_i}}} {1\over a_0}e^{-\mathcal L_V}\tilde\Pi(x)
\end{eqnarray*}
where in the formula we have identified $a_k\equiv a_{v^*_k}$ for convenience, and $a_{v_k^*+v_i^*}$ just corresponds to the point  $v_k^*+v_i^*$.
Note that
\begin{eqnarray*}
   & a_{v_i^*}\left\{ \sum_{v_k^*\in \Delta^*}(<v_k^*, v_{\rho_i}>+1)a_{v^*_k}{\pa\over \pa a_{v^*_k+v_i^*}} \right\}{1\over a_0}&\\=&\left\{ \sum_{v_k^*\in \Delta^*}(<v_k^*, v_{\rho_i}>+1){a_{v^*_k}a_{v_i^*}\over a_0 a_{v_i^*+v_k^*}}\left(a_{v^*_k+v_i^*} {\pa\over \pa a_{v^*_k+v_i^*}} \right)\right\}-2 {a_{v_i^*}a_{-v_i^*}\over a_0^2}&
\end{eqnarray*}
where the last term is zero if $-v_i^*\neq \Delta^*$. On the other hand, it's easy to compute
\begin{eqnarray*}
    e^{\mathcal L_V}\mathcal L_{\left\{\prod\limits_{\rho=1}^5 z_\rho^{<v_i^*, v_\rho>}\right\}z^{\rho_i}{\pa\over \pa z^{\rho_i}}} e^{-\mathcal L_V}=\prod\limits_{\rho=1}^5 \left( a_0\over a_\rho \right)^{{w_\rho\over d}<v_i^*,v_\rho>}\mathcal L_{\left\{\prod\limits_{\rho=1}^5 z_\rho^{<v_i^*, v_\rho>}\right\}z^{\rho_i}{\pa\over \pa z^{\rho_i}}}
\end{eqnarray*}
therefore we get the following
\begin{eqnarray*}
\left\{ \left\{ \sum_{v_k^*\in \Delta^*}(<v_k^*, v_{\rho_i}>+1){a_{v^*_k}a_{v_i^*}\over a_0 a_{v_i^*+v_k^*}}\left(a_{v^*_k+v_i^*} {\pa\over \pa a_{v^*_k+v_i^*}} +
\delta_{1\leq v_i^*+v_k^*\leq 5}\mathcal L_{{w_i\over d}z_{v_i^*+v_k^*}{\pa\over \pa z_{v_i^*+v_k^*}}}
\right)\right\}-2 {a_{v_i^*}a_{-v_i^*}\over a_0^2}   \right\}\tilde \Pi(x)\\=\left({a_{v_i^*}\over a_0}\right)\prod_{\rho} \left( a_0\over a_\rho \right)^{{w_\rho\over d}<v_i^*,v_\rho>}\mathcal L_{\left\{\prod_{\rho}z_\rho^{<v_i^*, v_\rho>}\right\}z^{\rho_i}{\pa\over \pa z^{\rho_i}}} \tilde \Pi(x)
\end{eqnarray*}
Again, by the Cartan-Lie formula we can easily write the right side as a $d$-exact form.

\subsection*{Example: $\bl{P}(1,1,1,1,1)$} 

We will compute the $\beta$-term for mirror quintic \cite{open mirror}, where $w=(1,1,1,1,1)$, and
$$
    f_{\Delta^*}=a_0\prod\limits_{i=1}^5 z_i+ \sum\limits_{i=1}^5 a_i z_i^5.
$$
The relation lattice is generated by
$$
    l=(-5,1,1,1,1,1)
$$
and the moduli variable is
\begin{eqnarray}.
    x=(-1)^{l_0}a^l=-{\prod\limits_{i=1}^5 a_i\over a_0^5}
\end{eqnarray}
Put
\begin{eqnarray}
    \tilde \Pi(x)={\omega\over \prod\limits_{i=1}^5 z_i- x^{1\over 5} \sum\limits_{i=1}^5 z_i^5}
\end{eqnarray}
Then our $\beta$-term formula Proposition \ref{beta term} for the Picard-Fuchs equation yields
\begin{eqnarray}\label{picard-fuchs}
    \left\{\prod\limits_{i=1}^5 \left( \Theta_x+{1\over 5}\mathcal L_{z_i\pa_{z_i}} \right)-x \prod\limits_{i=1}^5 \left(  5\Theta_x+i \right)  \right\} \tilde \Pi(x)=0
\end{eqnarray}
or
\begin{eqnarray*}
    \left(\Theta_x^5-x\prod\limits_{i=1}^5(5 \Theta_x+i)\right)\tilde \Pi(x)=-d\beta_x, \ \ d\beta_x=\sum_{I\subset \{1,\dots, 5\}, |I|\geq 1}{\Theta_x^{5-|I|}\over 5^{|I|}}\prod_{k\in I} (\mathcal L_{z_k\pa_k})\tilde \Pi(x).
\end{eqnarray*}
Note that the sum of terms for $|I|=1$ in the expression of $d\beta_x$ is zero. A choice of $\beta_x$ relevant in the next section will be
\begin{eqnarray*}
    \beta_x&=&{\Theta_x^3\over 5^2}\left\{ z_1\iota_{\pa_1} \mathcal L_{z_2\pa_2}+ z_3\iota_{\pa_3}\mathcal L_{z_4\pa_4} +(z_3\iota_{\pa_3}+z_4\iota_{\pa_4})(\mathcal L_{z_1\pa_1}+\mathcal L_{z_2\pa_{2}}) \right\}\tilde \Pi(x)\\
    &&+{\Theta_x^2\over 5^3}\left\{(z_3\iota_{\pa_3}+z_4\iota_{\pa_4})\mathcal L_{z_1\pa_1}\mathcal L_{z_2\pa_2} + z_3\iota_{\pa_3}(\mathcal L_{z_1\pa_1}+\mathcal L_{z_2\pa z_2})\mathcal L_{z_4\pa_4} \right\}\tilde \Pi(x)\\
    &&+ {\Theta_x \over 5^4} z_3\iota_{\pa_3}\mathcal L_{z_4\pa_4}\mathcal L_{z_1\pa_1}\mathcal L_{z_2\pa_2} \tilde \Pi(x)
    +\sum_{I\subset \{1,\dots, 4\}, |I|\geq 1}{\Theta_x^{4-|I|}\over 5^{|I|+1}}z_5\iota_{\pa_5}\prod_{k\in I} (\mathcal L_{z_k\pa_k})\tilde \Pi(x)
\end{eqnarray*}
where $\iota_{\pa_i}$ is the contraction with the vector ${\pa\over \pa z_i}$.

\subsection*{Example: $\bl{P}(2,2,2,1,1)$}
The mirror of degree 8 hypersurface in $\bl{P}(2,2,2,1,1)$ has two complex moduli. The integral points of its dual polytope $\Delta^*$ is given
by
\begin{eqnarray*}
     \Delta^*: && v^*_0=(0,0,0,0) \\
     && v^*_1=(1,0,0,0)\\ && v^*_2=(0,1,0,0)\\ &&
              v^*_3=(0,0,1,0)\\ && v^*_4=(0,0,0,1)\\ && v^*_5=(-2,-2,-2,-1)\\
              && v^*_6=(-1,-1,-1,0)
\end{eqnarray*}
and in homogeneous coordinate
\begin{eqnarray*}
    f_{\Delta^*}=a_0 z_1z_2z_3z_4z_5+a_1 z_1^4+ a_2 z_2^4+a_3 z_3^4+a_4 z_4^8+a_5 z_5^8+ a_6 z_4^4z_5^4
\end{eqnarray*}
A basis of relation lattice is given by
$$
    l^{(1)}=(-4,1,1,1,0,0,1), \ \ l^{(2)}=(0,0,0,0,1,1,-2)
$$
We get the moduli coordinates
$$
    x_1={a_1a_2a_3a_6\over a_0^4}, \ x_2={a_4a_5\over a_6^2}
$$
The rational 4-form is given by
\begin{eqnarray*}
    \tilde \Pi(x)={\Omega_0\over z_1z_2z_3z_4z_5+x_1^{1\over 4}x_2^{1\over 8}\left(z_1^4+z_2^4+z_3^4+z_4^8+z_5^8 \right)
    +x_1^{1\over 4}x_2^{-{3\over 8}}z_4^4z_5^4}
\end{eqnarray*}
The exact GKZ equation from $l^{(1)}$ and $l^{(2)}$ can be read
\begin{eqnarray*}
\left\{(\Theta_{x_1}-2\Theta_{x_2})\prod_{j=1}^3 (\Theta_{x_1}+{1\over 4}\mathcal L_{z_i\pa_{z_i}})   -x_1\prod_{j=1}^4(4\Theta_{x_1}+j) \right\}\tilde \Pi(x)&=&0\\
\left\{ \left(\Theta_{x_2}+{1\over 8}\mathcal L_{z_4\pa_{z_4}}\right)\left(\Theta_{x_2}+{1\over 8}\mathcal L_{z_5\pa_{z_5}}\right)
-x_2 \left(\Theta_{x_1}-2\Theta_{x_2}\right)\left(\Theta_{x_1}-2\Theta_{x_2}-1\right) \right\}\tilde \Pi(x)&=&0
\end{eqnarray*}

\subsection*{Example: $\bl{P}(7,2,2,2,1)$}
The mirror of degree 14 hypersurface in $\bl{P}(7,2,2,2,1)$ has also two complex moduli. The integral points of its dual polytope $\Delta^*$ is given
by
\begin{eqnarray*}
     \Delta^*: && v^*_0=(0,0,0,0) \\
     && v^*_1=(1,0,0,0)\\ && v^*_2=(0,1,0,0)\\ &&
              v^*_3=(0,0,1,0)\\ && v^*_4=(0,0,0,1)\\ && v^*_5=(-7,-2,-2,-2)\\
              && v^*_6=(-3,-1,-1,-1)\\
              && v^*_7=(-4,-1,-1,-1)\\
              && v^*_8=(-1,0,0,0)
\end{eqnarray*}
and in homogeneous coordinate
\begin{eqnarray*}
    f_{\Delta^*}=a_0 z_1z_2z_3z_4z_5+a_1 z_1^2+ a_2 z_2^7+a_3 z_3^7+a_4 z_4^7+a_5 z_5^{14}+ a_6 z_1z_5^7+a_7z_2z_3z_4z_5^8+a_8z_2^2
    z_3^2z_4^2z_5^2
\end{eqnarray*}
A basis of relation lattice is given by \cite{Hosono-Lian}
\begin{eqnarray*}
    l^{(1)}=(-1,0,0,0,0,-1,1,1,0), \ \ l^{(2)}=(0,1,0,0,0,1,-2,0,0)\\
    l^{(3)}=(0,0,1,1,1,0,0,1,-4), \ \ l^{(4)}=(0,0,0,0,0,1,0,-2,1)
\end{eqnarray*}
We get the "moduli" coordinates before eliminating the automorphism generated by the roots
$$
    x_1=-{a_6a_7\over a_0 a_5}, \ x_2={a_1a_5\over a_6^2}, \ x_3={a_2a_3a_4a_7\over a_8^4}, \ x_4={a_5a_8\over a_7^2}
$$
The rational 4-form is given by
\begin{eqnarray*}
    \tilde \Pi(x)={\Omega_0\over z_1z_2z_3z_4z_5-x_1 x_2^{1\over 2}x_3^{1\over 7}x_4^{4\over 7}\left(z_1^2+z_2^7+z_3^7+z_4^7+z_5^{14} \right)
    -x_1x_3^{1\over 7}x_4^{4\over 7}z_1z_5^7-x_1x_2^{1\over 2}z_2z_3z_4z_5^8-x_1x_2^{1\over 2}x_3^{-{1\over 7}}x_4^{3\over 7}z_2^2z_3^2z_4^2z_5^2}
\end{eqnarray*}
We consider the GKZ operator corresponding to the following two relation vectors (\cite{Hosono-Lian})
$$
    l_{\{1,5\}}=(0,1,0,0,0,1,-2,0,0), \ \ l_{\{2,3,4,6\}}=(-1,0,1,1,1,0,1,0,-3)
$$
which gives two exact GKZ equations
\begin{eqnarray*}
    \left\{ \left( \Theta_{x_2}+{1\over 2}\mathcal L_{z_1\pa_{z_1}} \right)\left(-\Theta_{x_1}+\Theta_{x_2} +\Theta_{x_4}+{1\over 14}
    \mathcal L_{z_5\pa_{z_5}}  \right)-x_2\left( \Theta_{x_1}-2\Theta_{x_2} \right)\left( \Theta_{x_1}-2\Theta_{x_2} -1 \right)  \right\}\tilde\Pi(x)&=&0\\
    \left\{\left( \Theta_{x_1}-2\Theta_{x_2} \right) \prod_{i=2}^4\left(\Theta_{x_3}+{1\over 7}\mathcal L_{z_i\pa_{z_i}} \right)
    -x_1x_3x_4 \left(\Theta_{x_1}+1\right)\prod_{i=0}^2\left( \Theta_{x_4}-4\Theta_{x_3}-i \right)  \right\}\tilde\Pi(x)&=&0
\end{eqnarray*}
and the roots $v_7^*, v_8^*$ give two exact equations
\begin{eqnarray*}
\left\{ \left(\Theta_{x_1}+\Theta_{x_3}-2\Theta_{x_4}\right)-2x_1x_2\left(\Theta_{x_1}-2\Theta_{x_2} \right)-x_1\left( -\Theta_{x_1}+\Theta_{x_2} +\Theta_{x_4}+{1\over 14}
    \mathcal L_{z_5\pa_{z_5}}  \right) +x_1x_2^{1\over 2} \mathcal L_{z_5^7\pa_{z_1}} \right\}\tilde\Pi(x)&=&0\\
    \left\{ \left( \Theta_{x_4}-4\Theta_{x_3} \right)-2 x_1^2x_2x_4(\Theta_{x_1}+1)-x_1x_4\left(\Theta_{x_1}+\Theta_{x_3}-2\Theta_{x_4}\right)
    +x_1x_2^{1\over 2}x_3^{-{1\over 7}}x_4^{3\over 7}\mathcal L_{z_2z_3z_4z_5\pa_{z_1}}  \right\}\tilde \Pi(x)&=&0
\end{eqnarray*}

\subsection*{Calabi-Yau hypersurfaces in a general toric variety}
For Calabi-Yau hypersurfaces of non-Fermat type, we can repeat the derivation of the $\beta$-term above by replacing the homogeneous coordinates by the toric coordinates $X_i$. Write the defining equation of a general Calabi-Yau hypersurface is in toric coordinates: 
\begin{eqnarray*}
    f_{\Delta^*}(X)=\sum_{v_k^*\in \Delta^*}a_k X^{v_k^*}.
\end{eqnarray*}
Here we will assume that $v_0^*=(0,0,0,0), v_1^*=(1,0,0,0), v_2^*(0,1,0,0), v_3^*=(0,0,1,0), v_4^*=(0,0,0,1)$. Put
\begin{eqnarray*}
    \Pi(a)={1\over f_{\Delta^*}(X)}\prod_{j=1}^4 {dX_j\over X_j}
\end{eqnarray*}
As in the Fermat case above, we use the automorphism
\begin{eqnarray*}
   \phi: X_i\to {a_0\over a_i} X_i, \ \ 1\leq i\leq 4
\end{eqnarray*}
to transform $\Pi(a)$ into a form parameterized only by moduli variables $x_k=(-1)^{l^{(k)}_0}a^{l^{(k)}}$. 
We can now repeat the derivation of the $\beta$-term for each GKZ operator $D_l$, using the form $\tilde \Pi(x)=a_0 \phi^* \Pi(a)$. To summarize, we have the following result.

\begin{prop}
Let $\tilde \Pi(x)=a_0 \phi^* \Pi(a)$, and for each $l\in L$, let
$$
\tilde D_l=\prod_{j=1}^{l_0}(a_0{\pa\over \pa a_0}-j) \prod_{i\neq 0,l_i>0} \prod_{j=0}^{l_i-1}(a_i{\pa\over \pa a_i}-j)-a^l \prod_{j=1}^{-l_0}(a_0{\pa\over \pa a_0}-j) \prod_{i\neq 0, l_i<0} \prod_{j=0}^{-l_i-1} (a_i{\pa\over \pa a_i}-j).
$$
Then $\tilde \Pi(x)$ is killed by the differential operator obtained from $\tilde D_l$ by the substitutions:
\begin{eqnarray*}
    a_0{\pa\over \pa a_0}&\mapsto&  a_0{\pa\over \pa a_0}-\mathcal L_{\sum\limits_{j=1}^4 X_j{\pa\over \pa X_j}}\\
    a_j{\pa\over \pa a_j}&\mapsto& a_j{\pa\over \pa a_j}+\mathcal L_{X_j{\pa\over \pa X_j}}, \ \ \ 1\leq j\leq 4\\
    a_j{\pa\over \pa a_j}&\mapsto& a_j{\pa\over \pa a_j}, \ \ \ j>4.
\end{eqnarray*}
\end{prop}
From this, the Cartan-Lie formula then explicitly yields a $\beta$-term for each GKZ operator:
$$
\tilde D_l\tilde \Pi(x)=-d\beta_l.
$$

\section{Application to Abel-Jacobi Map And Inhomogeneous Picard-Fuchs Equation}
In this section, we apply the inhomogeneous Picard-Fuchs equation Proposition \ref{Inhomog} to study the Abel-Jacobi map for toric Calabi-Yau hypersurfaces that arise in open string mirror symmetry \cite{D-brane Normal Function}. We will derive a purely algebraic double residue formula for computing the boundary integral of the $\beta$-term. \\

We will keep the same notation as in section 1: $X_z$ will be a
family of hypersurfaces moving in fixed 4-dimensional ambient space $M$
parametrized by $z$, and we consider two fixed hypersurfaces $Y_1:
Q_1=0, Y_2: Q_2=0$, where $Q_1, Q_2$ are sections of some fixed
divisors. We consider a pair of family of curves $C^+_z, C^-_z$ such
that they are homologically equivalent as cycles
$$
    C^{\pm}_z\hookrightarrow X_z\cap Y_1\cap Y_2, \ \
    [C^+_z]=[C^-_z]\in H_2(X_z, \mathbb Z)
$$
We denote by $\mathcal C^{\pm}$ he corresponding families
\begin{eqnarray*}
    \xymatrix{
        \mathcal C^\pm \ar[dr]\ar@{^{(}->}[r] & \mathcal X\ar[d]\\
        & S
    }
\end{eqnarray*}
Let $(\mathcal H_{\mathbb Z}^3, F^*\mathcal H_{\mathbb C}^3)$ be the
integral Hodge structure for the family $\mathcal X\to S$, and
$\mathcal J^3$ be the intermediate Jacobian fibration
\begin{eqnarray*}
    \mathcal J^3={\mathcal H_{\mathbb C}^3\over F^2\mathcal H_{\mathbb C}^3\oplus \mathcal H_{\mathbb
    Z}^3}
\end{eqnarray*}
The fiber of $\mathcal J^3$ over $z$ can be identified with
$$
J^3_z={\left( F^2 H^3(X_z, \mathbb C) \right)^*\over H_{3}(X_z,
\mathbb Z)}.
$$
The Abel-Jacobi map associated to $\mathcal C^{\pm}$ is a normal
function  of $\mathcal J^3$ given by
\begin{eqnarray*}
    \int_{C^-_z}^{C^+_z}: S\to \mathcal J^3.
\end{eqnarray*}
Suppose we are given a family of rational $4$-form with pole of order
one along $X_z$
$$
    \omega_z\in H^0(M, K_M(X_z)).
$$
Then we obtain a section $ Res_{X_z} \omega_z$ of $F^3\mathcal
H^3_{\mathbb C}$. The integral that we will study is
$$
    \int_{C^-_{z}}^{C^+_z} Res_{X_z} \omega_z,
$$
which is well-defined up to periods of closed cycles on $X_z$. \\

We can assume that $C_z^+\cup
C_z^-$ moves equisingularly above $S$, for otherwise we shrink $S$ to
some smaller open subset to achieve this. We can always choose local
trivialization of $\mathcal X$ over a small disk around any point in
$S$ which also trivializes $\mathcal C^+\cup \mathcal C^-$.
Therefore it's easy to see that $H^3(X_z, Y_z)$ still forms a local
system in this case by $Y_z=C_z^+\cup C_z^-$ and the discussion in
section one is still valid. Let $\{p_{z,A}\}$ be the singular points
of $Y_1\cap Y_2\cap X_z$ corresponding to non-transversal
intersections. It contains $C^+_z\cap C^-_z$ and possibly some
other points on $C^+_z$ or $C^-_z$ when there are components of
$Y_1\cap Y_2$ other than $C^\pm_z$. {\it We assume that at each $p_{z,A}$,
one of $Y_1, Y_2$ is intersecting transversely with $X_z$.} Therefore
we can always put an $\epsilon$-tube around each of $C^{\pm}_z$ such that the $\epsilon$-tube lies
on $Y_1\cap Y_2$ outside a small disk centered at
each $p_{z, A}\in C^{+}_z\cup C^-_z$, where the $\epsilon$-tube lies in one of $Y_1$ or $Y_1$ around that small disk.

Suppose $\mathcal D=\mathcal D(\pa_z)$ is a Picard-Fuchs operator,
and
\begin{eqnarray*}
    \mathcal D(\pa_z) \omega_z=-d\beta_z
\end{eqnarray*}
Applying (\ref{localization}) and using the fact that $\beta_z$ is
of type $(3,0)$, we get
\begin{eqnarray*}
    \mathcal D(\pa_z) \int_{C^-_z}^{C^+_z} Res_{X_z} {\omega_z}&=&{1\over 2\pi
    i}\int_{T_\epsilon(C^+_z-C^-_z)} \beta_z\\
    &=&\sum_{p_{z, A}\in C^+_z} {1\over 2\pi
    i}\int_{T_\epsilon(D^+_{z,A})} \beta_z-\sum_{p_{z, A}\in C^-_z} {1\over 2\pi
    i}\int_{T_\epsilon(D^-_{z,A})} \beta_z
\end{eqnarray*}
where $D^{\pm}_{z,A}$ is a small disk around $p_{z,A}\in
C^{\pm}_z$ such that $T_{\epsilon}(\pa D^{\pm}_{z,A})\in Y_1\cap
Y_2$. 

Fix a $p_{z, A}\in C^+_z$ and let $P_{z,A}=0$ be the local defining equation for
$X_z$ near $p_{z,A}$. We assume that $T_\epsilon(D^+_{z,A})\subset Y_1$. Then
$\beta_z$ restricts to $Y_1$ becoming a top $(3,0)$-form, and we can write
\begin{eqnarray}
    \beta_z|_{Y_1}=d \eta_{z,A}
\end{eqnarray}
for some local holomorphic $(2,0)$-form $\eta_{z,A}$ on $Y_1$, defined near
$p_{z,A}$ and with poles along $Y_1\cap X_z$. Hence
\begin{eqnarray*}
{1\over 2\pi
    i}\int_{T_\epsilon(D^+_{z,A})} \beta_z&=&-{1\over 2\pi
    i}\int_{T_\epsilon(\pa D^+_{z,A})} \eta_{z,A}\\
    &=&-\int_{\pa D^+_{z,A}} Res_{C^+_{z,A}}
    (\eta_{z,A}|_{Y_1\cap Y_2})\\
    &=&-2\pi i\ Res_{p_{z,A}} Res_{C^+_z}
    \left(\eta_{z,A}|_{Y_1\cap Y_2}\right)
\end{eqnarray*}
Thus we arrive at the following ``localization'' formula
\begin{thm} [Double Residue Formula]
Under the stated assumption above, we have
\begin{eqnarray}
\mathcal D(\pa_z) {1\over 2\pi i}\int_{C^-_z}^{C^+_z}
Res_{X_z}\omega_z&=&-\sum_{p_{z, A}\in C^+_z} Res_{p_{z,A}}
Res_{C^+_z}
    \left(\eta_{z,A}|_{Y_1\cap Y_2}\right)\nonumber\\
    &&+\sum_{p_{z, A}\in C^-_z}  Res_{p_{z,A}} Res_{C^+_z}
    \left(\eta_{z,A}|_{Y_1\cap Y_2}\right)\label{double residue}
\end{eqnarray}
\end{thm}
Note that in this double residue formula, we can throw away those
terms in $\eta_{z,A}$ with pole of order one along $X_{z}$. In fact,
since $X_z$ and $Y_1$ intersects transversely locally at $p_{z,A}$,
we can write locally
$$
    \beta_z|_{Y_1}=d(\xi_{z,A})+\alpha_{z,A}
$$
where $\alpha_{z,A}$ has pole of order one along $X_z\cap Y_1$.
Since $d\alpha_{z,A}=0$, we can furthermore write
$$
\alpha_{z,A}=d(\alpha^\prm_{z,A})
$$
where $\alpha^\prm_{z,A}$ is logrithmic along $X_z\cap Y_1$.  Then
$$
    Res_{C_z^+}\alpha^\prm_{z,A}
$$
will be holomorphic on $C_z^+$ around $p_{z,A}$, which is killed by taking another residue at $p_{z,A}$. This
leads to the following algebraic algorithm to calculate $\eta_{z,A}$
in practice: if locally around $p_{z,A}$ inside $Y_1$ we have
$$
    \beta_z|_{Y_1}={d P_{z,A}\over P_{z,A}^l}\wedge \zeta, \ \ \ l>1
$$
where $\zeta$ has no pole along $X_z\cap Y_1$, then we can reduce
the order of pole by
$$
    \beta_z|_{Y_1}=-d\left( {\zeta\over (l-1)P_{z,A}^{l-1}}
    \right)+{d\zeta\over (l-1)P_{z,A}^{l-1}}
$$
and apply this procedure until we get
$$
\beta_z|_{Y_1}= d(\psi_{z,A})+\varphi_{z,A}
$$
where $\psi_{z,A}$ is of the form ${\zeta_{l-1}\over
P_{z,A}^{l-1}}+{\zeta_{l-2}\over
P_{z,A}^{l-2}}+\cdots+{\zeta_{1}\over P_{z,A}}$, and $\varphi_{z,A}$
has pole of order one. Then we can simply use $\psi_{z,A}$ instead
of $\eta_{z_A}$ in (\ref{double residue}) to compute the double residue.

\subsection*{Example: Mirror Quintic}
We will apply formula (\ref{double residue}) to the example studied
in \cite{D-brane Normal Function} where the inhomogeneous Picard-Fuchs equation for the Abel-Jacobi map was
computed by using a much more elaborate analytic method.\\
\\
 Consider the one-parameter family of quintic in $\mathbb P^4$ defined by
\begin{eqnarray}
    X_x=\{P_x= \prod\limits_{i=1}^5 z_i-x^{1/5}\sum\limits_{i=1}^5 z_i^5=0\}
\end{eqnarray}
Let $\omega=\iota_\theta\lrcorner \Omega,\ \Omega=dz_1\wedge
dz_2\wedge\cdots\wedge dz_5,\ \theta=\sum\limits_{i=1}^5 z_i
{\pa\over \pa z_i}$, and the family of holomorphic 3-form on $X_x$
is given by the residue map
$$
            \Omega_x=Res_{P_x} {\omega\over P_x}
$$
and consider a pair of families of curves
$$
    C_{x}^{\pm}=\{z_1+z_2=0, z_3+z_4=0, z_1z_3\pm \sqrt{x^{1/5}}z_5^2=0\}
$$
We denote
$$
Y_1=\{z_1+z_2=0\}, \ \ \ Y_2=\{z_3+z_4=0\}
$$
then $Y_1$ is transversal with $P_x$ except at $p_1$, and $Y_2$ is
transversal with $P_x$ except at $p_2$, where
$$
    p_1=[1,-1,0,0,0], \ \ p_2=[0,0,1,-1,0]
$$
We consider the Picard-Fuchs operator from the GKZ-system and by
(\ref{picard-fuchs})
$$
    D(\pa_x)= \Theta_x^5-x\prod\limits_{i=1}^5(5 \Theta_x+i), \ \ d\beta_x=\sum_{I\subset \{1,\dots, 5\}, |I|\geq 1}{\Theta_x^{5-|I|}\over 5^{|I|}}\prod_{k\in I} (\mathcal L_{z_k\pa_k})\tilde \Pi(x)
$$
We choose $\beta$ to be
\begin{eqnarray*}
    \beta&=&{\Theta_x^3\over 5^2}\left\{ z_1\iota_{\pa_1} \mathcal L_{z_2\pa_2}+ z_3\iota_{\pa_3}\mathcal L_{z_4\pa_4} +(z_3\iota_{\pa_3}+z_4\iota_{\pa_4})(\mathcal L_{z_1\pa_1}+\mathcal L_{z_2\pa_{2}}) \right\}\tilde \Pi(x)\\
    &&+{\Theta_x^2\over 5^3}\left\{(z_3\iota_{\pa_3}+z_4\iota_{\pa_4})\mathcal L_{z_1\pa_1}\mathcal L_{z_2\pa_2} + z_3\iota_{\pa_3}(\mathcal L_{z_1\pa_1}+\mathcal L_{z_2\pa z_2})\mathcal L_{z_4\pa_4} \right\}\tilde \Pi(x)\\
    &&+ {\Theta_x \over 5^4} z_3\iota_{\pa_3}\mathcal L_{z_4\pa_4}\mathcal L_{z_1\pa_1}\mathcal L_{z_2\pa_2} \tilde \Pi(x)
    +\sum_{I\subset \{1,\dots, 4\}, |I|\geq 1}{\Theta_x^{4-|I|}\over 5^{|I|+1}}z_5\iota_{\pa_5}\prod_{k\in I} (\mathcal L_{z_k\pa_k})\tilde \Pi(x)
\end{eqnarray*}
Then the calculation of Abel-Jacobi map is localized near $p_1, p_2$
by formula (\ref{double residue}).\\

\noindent $\bullet$ Near $p_2$. We have
$$
    \beta|_{Y_1}={\Theta_x^3\over 5^2}\left(z_1\iota_{\pa_1} \mathcal L_{z_2\pa_2}\tilde \Pi(x)  \right)|_{Y_1}
$$
Note that we can move $\Theta_x$ outside the integral in the local
calculation of double residue. If we choose local coordinate such
that $z_3=1, z_2=-z_1$, then
\begin{eqnarray*}
\left(z_1\iota_{\pa_1} \mathcal L_{z_2\pa_2}\tilde \Pi(x)  \right)|_{Y_1}&=&-{x^{1/5}z_1(1+z_4^5+z_5^5+5 z_1^5) dz_4dz_1dz_5\over \left[
    -x^{1/5}(1+z_4^5+z_5^5)-z_1^2z_4z_5
\right]^2}\\
&=&d\left( {x^{1/5}z_1(1+z_4^5+z_5^5+5 z_1^5) dz_1\wedge dz_5\over (-5x^{1/5}z_4^4-z_1^2z_5) \left[
    -x^{1/5}(1+z_4^5+z_5^5)-z_1^2z_4z_5
\right]} \right)+\mbox{order one term}
\end{eqnarray*}
Then the relevant residue is given by
\begin{eqnarray*}
    Res_{p_2}Res_{C_z^{\pm}}\left( {x^{1/5}z_1(1+z_4^5+z_5^5+5 z_1^5) dz_1\wedge dz_5\over (-5x^{1/5}z_4^4-z_1^2z_5) \left[
    -x^{1/5}(1+z_4^5+z_5^5)-z_1^2z_4z_5
\right]} \right)|_{z_4=-1}
\end{eqnarray*}
which is easily computed to be zero. \\

\noindent $\bullet$ Near $p_1$. We have
\begin{eqnarray*}
  \beta|_{Y_2}={\Theta_x^3\over 5^2}\left(z_3\iota_{\pa_3}\mathcal L_{z_4\pa_4}\tilde\Pi(x)\right)
  +{\Theta_x^2\over 5^3}d\left((\iota_{z_1\pa_1}+\iota_{z_2\pa z_2})z_3\iota_{\pa_3}\mathcal L_{z_4\pa_4}  \tilde\Pi(x)\right)
+{\Theta_x \over 5^4}
d\left(\iota_{z_1\pa_1}\iota_{z_3\pa_3}\mathcal L_{z_4\pa_4}\mathcal
L_{z_2\pa_2} \tilde \Pi(x)\right)
\end{eqnarray*}
The first term will be zero after double residue by the same calculation as above. The
contribution of the second term is also zero since
$$
\left((\iota_{z_1\pa_1}+\iota_{z_2\pa z_2})z_3\iota_{\pa_3}\mathcal
L_{z_4\pa_4}  \tilde\Pi(x)\right)|_{Y_1\cap Y_2}=0
$$
Therefore we get
\begin{eqnarray*}
    D(\pa_x){1\over 2\pi i}\int_{C^-_x}^{C^+_x} \Omega_x&=&-{\Theta_x\over 5^4}Res_{p_1}
Res_{C^+_z}
   \left(\iota_{z_1\pa_1}\iota_{z_3\pa_3}\mathcal L_{z_4\pa_4}\mathcal
L_{z_2\pa_2} \tilde \Pi(x)\right)|_{Y_1\cap Y_2}\\&&+ {\Theta_x\over
5^4}Res_{p_1} Res_{C^-_z}
   \left(\iota_{z_1\pa_1}\iota_{z_3\pa_3}\mathcal L_{z_4\pa_4}\mathcal
L_{z_2\pa_2} \tilde \Pi(x)\right)|_{Y_1\cap Y_2}
\end{eqnarray*}
The double residue now is easy to compute and we get
\begin{eqnarray*}
Res_{p_1} Res_{C^{\pm}_x}
   \left(\iota_{z_1\pa_1}\iota_{z_3\pa_3}\mathcal L_{z_4\pa_4}\mathcal
L_{z_2\pa_2} \tilde \Pi(x)\right)|_{Y_1\cap Y_2}=\mp {375\over
8}\sqrt{x}
\end{eqnarray*}
therefore
\begin{eqnarray}\label{quintic AJ}
D(\pa_x){1\over 2\pi i}\int_{C^-_x}^{C^+_x} \Omega_x=\Theta_x{3\over
20}\sqrt{x}
\end{eqnarray}

To compare our result with \cite{D-brane Normal Function}, we use
the normalization
\begin{eqnarray*}
    \hat\Omega_x&=&-\left(5\over 2\pi i\right)^3 \Omega_x, \\
     \mathcal
    L&=&\Theta_x^4-5x
    (5\Theta_x+1)(5\Theta_x+2)(5\Theta_x+3)(5\Theta_x+4)\\
    D(\pa_x)&=&\Theta_x\mathcal L
\end{eqnarray*}
then we have
\begin{eqnarray}
    \Theta_x \mathcal L \int_{C^-_x}^{C^+_x}\hat \Omega_x=\Theta_x{75\over 16 \pi^2}\sqrt{x}
\end{eqnarray}
It follows that
$$
\mathcal L \int_{C^-_x}^{C^+_x}\hat \Omega_x={75\over 16
\pi^2}\sqrt{x}+c
$$
for some constant $c$. Since the monodromy $x\to e^{2\pi i}x$ will
switch the curves $C^+_x$ and $C^-_x$, the constant $c$ is zero. We
therefore obtain the inhomogeneous term differs from the result of
\cite{D-brane Normal Function} by a factor of $5$, which is exactly
the order of stabilizer of $C^\pm_x$ under the finite quotient that yields
the mirror quintic.

\section{GKZ System and Picard-Fuchs Equation for Relative Cohomology}
In this section, $W$ will be fixed ambient space,
$X_{z}=\{P_{z}=0\}$ a family of hypersurfaces parametrized by $z$,
and $D_u=\{Q_u=0\}$ a family of hypersurfaces parametrized by $u$. Here $P_z, Q_u$ moves in the linear systems of certain possibly different
fixed 
line bundles on $W$. 
Let $Y_{z,u}\subset X_{z}$ be the intersection $X_{z}\cap D_u$, and
we consider the local system of relative cohomology
$$
    H^3(X_z, Y_{z,u})
$$
parametrized by both $z$ and $u$. Given a rational form $\omega_z$
in $W$ with pole of order one along $X_z$, and choose a smooth
family of relative cycles $\Gamma_{z,u}\in H_3(X_z, Y_{z,u})$. We
will study the Picard-Fuchs equation for the relative period
\begin{eqnarray}
   \Pi(z,u)= \int_{\Gamma_{z,u}} Res_{X_z}\omega_z.
\end{eqnarray}
We can assume that the relative cycle $\Gamma_{z,u}$ is away from any non-transversal intersection point of $X_z, Q_u$, since we can always move the relative cycle away from those points without changing its cohomology class. \\

We first consider the variation with respect to $z$. We choose
$\epsilon$-tube of $\Gamma_{z,u}$ such that $\pa
T_\epsilon(\Gamma_{z,u})$ lies inside $D_u$. As shown in section 2, we have
\begin{eqnarray}\label{differentiation}
\left(\pa_z\right)^k \int_{\Gamma_{z,u}} Res_{X_z}\omega_z
= \left(\pa_z\right)^k {1\over 2\pi i}\int_{T_\epsilon(\Gamma_{z,u})}\omega_z=
{1\over 2\pi i} \int_{T_\epsilon(\Gamma_{z,u})}\left(\pa_z\right)^k \omega_z
\end{eqnarray}
and the above integration is independent of $\epsilon$. Let $\mathcal D_z$ be a Picard-Fuchs differential operator and let
$$
    \mathcal D_z \omega_z=-d\beta_z.
$$
Then formula (\ref{differentiation}) yields
\begin{eqnarray}
    \mathcal D_z \Pi(z,u)&=&
   {1\over 2\pi i} \int_{T_\epsilon(\Gamma_{z,u})} -d\beta_z={1\over 2\pi i}\int_{T_{\epsilon}(\pa\Gamma_{z,u})} \beta_z|_{D_u} \nonumber\\
      &=&\int_{(\pa\Gamma_{z,u})} Res_{Y_{z,u}}\beta_z|_{D_u} =\int_{(\pa\Gamma_{z,u})} Res_{Y_{z,u}} Res_{D_u}d\log Q_u\wedge \beta_z
    \end{eqnarray}
Note that $d\log Q_u
    \wedge \beta_z$ is not well-defined global form, but $Res_{D_u} d\log Q_u
    \wedge \beta_z$ is well-defined.\\

Next we fix $z$ and consider the variation with respect to $u$. Choose a lifting of $\pa\over \pa u$ to $\mathcal
Y$
\begin{eqnarray*}
    {\pa\over \pa u}+ n_{\mathcal Y_{z,u}}
\end{eqnarray*}
where $n_{\mathcal Y_{z,u}}$ is the normal vector field
corresponding to the deformation of $Y_{z,u}$ inside $X_z$ with
respect to $u$. Then using (\ref{variation res}) and keeping the same notation as in section 2, we get
\begin{eqnarray*}
    \pa_u \Pi(z,u)=\int_{\pa \Gamma_{z,u}}\iota_{n_{\mathcal
    Y_{z,u}}}\lrcorner \phi_z|_{X_z}
\end{eqnarray*}
Since
$$
    \pa_u Q_u|_{Y_{z,u}}=-
    \left(\iota_{n_{\mathcal
    Y_{z,u}}}\lrcorner d_M Q_u
    \right)|_{Y_{z,u}}
$$
and $\phi_z$ is holomorphic 3-form, we have
\begin{eqnarray*}
    \left(\iota_{n_{\mathcal
    Y_{z,u}}}\lrcorner
    \phi_z|_{X_z}\right)|_{Y_{z,u}}&=&-Res_{Y_{z,u}}((\pa_u\log
    Q_{u})\phi_z)|_{X_z}\\
    &=&-Res_{Y_{z,u}}Res_{X_z} (\pa_u \log Q_u)
    \omega_z\\
    &=&Res_{Y_{z,u}}Res_{Q_u}(\pa_u \log Q_u)
    \omega_z.
\end{eqnarray*}
It follows that
\begin{eqnarray}
  \pa_u \Pi(z,u)=\int_{\pa{\Gamma_{z,u}}}Res_{Y_{z,u}}Res_{Q_u}(\pa_u \log
  Q_u)
    \omega_z
\end{eqnarray}
Note that $(\pa_u \log
  Q_u)
    \omega_z$ is globally well-defined rational form with pole along
    $X_z\cup D_u$. To summarize, we have
    
\begin{thm}[Variations of Relative Periods] \label{open-close variation} Let $X_z$, $Y_{z,u}$, $\Gamma_{z,u}$, $D_u$, $\mathcal D_z$ and $\beta_z$ be as given above. Then the relative period $\Pi(z,u)= \int_{\Gamma_{z,u}} Res_{X_z}\omega_z$ satisfies the equations
\begin{eqnarray*}
  \pa_z \Pi(z,u)&=&\int_{T_\epsilon(\Gamma_{z,u})}{\pa_z}\omega_z, \ \ \ \mbox{where}\ T_{\epsilon}
  (\pa\Gamma_{z,u})\subset D_{u}=\{Q_u=0\} \\
 \pa_u \Pi(z,u)&=&\int_{\pa{\Gamma_{z,u}}}Res_{Y_{z,u}}Res_{D_u}(\pa_u \log
  Q_u)
    \omega_z\\
    \mathcal D_z \Pi(z,u)&=&
    \int_{\pa\Gamma_{z,u}} Res_{Y_{z,u}}\beta|_{Q_u}=\int_{\pa\Gamma_{z,u}} Res_{Y_{z,u}}Res_{D_u} d\log Q_u
    \wedge \beta_z.
\end{eqnarray*}
\end{thm}

\subsection*{Griffith-Dwork Reduction}
Notice that if we pick a fixed $Q_0$, then we can formally write
\begin{eqnarray*}
    Res_{D_u}d\log Q_u
    \wedge \beta_z&=&Res_{D_u}\left( d\left( \log\left({Q_u\over Q_0}\right)\ \beta_z \right)-\log {Q_u\over Q_0}\
    d\beta_z\right)\\&=&Res_{D_u}\left(d\left( \log\left({Q_u\over Q_0}\right)\ \beta_z \right)+\log\left({Q_u\over Q_0}\right)\ \mathcal D_z
    \omega_z\right)
\end{eqnarray*}
then we get
\begin{eqnarray*}
\mathcal D_z \Pi(z,u)&=&\int_{\pa\Gamma_{z,u}}
Res_{Y_{z,u}}Res_{D_u}\mathcal D_z\left( \log\left({Q_u\over
Q_0}\right)\ \mathcal
    \omega_z
\right)\\
\pa_u \Pi(z,u)&=&\int_{\pa\Gamma_{z,u}} Res_{Y_{z,u}}Res_{D_u}
\pa_u\left( \log\left({Q_u\over Q_0}\right)\ \mathcal
    \omega_z
\right)
\end{eqnarray*}
Therefore one can derive a Picard-Fuchs equation for $\Pi(z,u)$
using Griffith-Dwork reduction procedure \cite{Rational Integral} by
starting with $\log\left({Q_u\over Q_0}\right)\ \mathcal
    \omega_z$. This procedure of adding Log-$Q$ is proposed in physics
   recently \cite{Masoud}. The theorem above shows that such a procedure is mathematical justified. We
   refer to \cite{Masoud} for applications of this procedure to specific examples.
   
\subsection*{Enhanced Polytope and GKZ System}

We now apply the method developed in the preceding section to study toric B-branes on toric
Calabi-Yau hypersurfaces. For physical motivations of B-brane geometry
on Calabi-Yau manifolds and their mirror transformation, see
\cite{Mina Vafa}. The GKZ-system associated with an enhanced polytope as a way to understand toric B-brane geometry is proposed in \cite{Mayr}.
We will give a uniform derivation of the general enhanced GKZ system by using the variation formula we have developed earlier.\\

We consider Calabi-Yau hypersurfaces in four dimensional toric variety $M$. The results extend easily to arbitrary dimensions. As in section 2, we assume that the anti-canonical sections $H^0(M,K_M^{-1})$ has a monomial basis in toric coordinates $X_i$, whose exponents are exactly the set of integral points of an integral polytope $\Delta^*$.  
The  defining equation for a Calabi-Yau hypersurface $X_a^*$ in $M$ in the toric coordinates has the form
\begin{eqnarray}
    f_{\Delta^*}(a)=\sum_{v_i^*\in \Delta^*} a_i
    X^{v_i^*}
\end{eqnarray}
where $v_0^*$ is the origin by convention, and the relevant rational
form with pole along the hypersurface is given by
$$
    \omega_a={1\over f_{\Delta^*}(a)}\prod_{i=1}^4 {dX_i\over X_i}
$$
Let
$$
    L=\{ l=(l_0,l_1,...)\in \mathbb Z^{|\Delta^*|} | \sum_i l_i \bar v_i^*=0 \}, \ \ \mbox{where}\ \ \bar v_i^*=(1,v_i^*)
$$
be the relation lattice. Next we briefly recall the so-called B-brane geometry introduced in physics.
 In order to describe the so-called superpotential in GLSM \cite{phases}\cite{mirror symmetry}\cite{Mina Vafa}, one introduces an extra variable $P$, and let 
$$
 \bar X=(P, X)=(P, X_1, X_2, X_3, X_4).
$$
Obviously we have 
\begin{eqnarray}
    \sum_{i}l_i=0, \ \ \prod_{v_i^*\in
    \Delta^*}\left(\bar X^{\bar v_i^*}\right)^{l_i}=1, \ \ l\in L.
\end{eqnarray}
The lattice $L$ (denoted as  ``$Q$'' in \cite{Mina Vafa}) is thought of as the toric data that encodes the Calabi-Yau geometry. The ``brane'' data is  specified by one additional lattice vector {\it not} in $L$:
\begin{eqnarray}
    q=\{q_i\}_{v^*_i\in \Delta^*}\in\mathbb Z^{|\Delta^*|}, \ \ \sum_{i}q_i=0.
\end{eqnarray}
(In principle, one can consider a brane configuration involving more than one lattice vectors. But for simplicity we will consider the case of only one such vector.) The B-brane corresponding to a choice of B-brane vector $q$ is the divisor $Y_{a,b}^*=X_a^*\cap D_b$, where $D_b$ is the closure in $M$ of the locus
\begin{eqnarray}
 h(b):=b_0+b_1\prod_{v_i^*\in \Delta^*} \left(X^{v_i^*}\right)^{q_i}=0.
\end{eqnarray}
Here $[b_0,b_1]$ is called the open string modulus. Under open string mirror symmetry, the B-model data $(L,q)$ corresponds to certain A-model data that includes a special lagrangian cycle and a mirror manifold \cite{Mina Vafa}. 
\\

Next, we describe Picard-Fuchs equations for the relative periods of the family of pairs $(X_a^*,Y_{a,b}^*)$. Almost the entire {\it general theory} of periods in closed string mirror symmetry \cite{HKSY}\cite{Hosono-Lian} turns out to carry over to the current open string context, with few modifications.
For a given relative cycle $\Gamma_{a,b}\in H_3(X_a^*,Y_{a,b}^*)$, the corresponding relative period is
\begin{eqnarray}
    \Pi(a,b)=\int_{\Gamma_{a,b}} {1\over f_{\Delta^*}(a)}\prod_{i=1}^4 {dX_i\over X_i}. 
\end{eqnarray}
Consider the automorphism on $M$ given by the torus action
$$
    \phi(\lambda): X_k\to \lambda_k X_k, \ k=1,..,4.
$$
This induces a group action on the family of pairs $(X_a,Y_{a,b})$, hence on the parameter space. The
induced transformation on the parameters $(a,b)$ is given by
$$
  a_i\to \left(\prod_{k=1}^4 \lambda_k^{v_{i,k}^*} \right)a_i, \
  b_0\to b_0, \ b_1\to \left(\prod_{v_i^*\in \Delta^*}\left(\prod_{k=1}^4 \lambda_k^{v^*_{i,k}}
  \right)^{q_i}\right)b_1.
$$
Since the relative periods are invariant under the transformation, it follows that
\begin{eqnarray}
    \Pi\left(\left(\prod_{k=1}^4 \lambda_k^{v_{i,k}^*} \right)a_i,b_0,\left(\prod_{v_i^*\in \Delta^*}\left(\prod_{k=1}^4 \lambda_k^{v^*_{i,k}}
  \right)^{q_i}\right)b_1\right)=\Pi(a,b)
\end{eqnarray}
or equivalently
\begin{eqnarray}
       \sum_{v_i^*\in \Delta^*} v^*_{i,k}\left(a_i{\pa\over \pa
       a_i}+q_ib_1{\pa\over \pa b_1}\right) \Pi(a,b)=0, \ \ k=1,..,4.
\end{eqnarray}
Put
\begin{eqnarray}\label{Lk}
    \mathcal L_k&=&\sum_{v_i^*\in \Delta^*} \bar v^*_{i,k}\left(a_i{\pa\over \pa
       a_i}+q_ib_1{\pa\over \pa b_1}\right)-\beta_k,  \ k=0,1,..,4
\end{eqnarray}
where $\beta=(-1,0,0,0,0)$. Then we get
\begin{eqnarray}
    \mathcal L_k \Pi(a,b)=0
\end{eqnarray}
 Given $l\in L$, let
$$
    \mathcal D_l=\prod_{l_i>0}\left({\pa\over \pa a_i}\right)^{l_i}-\prod_{l_i<0}\left({\pa\over \pa a_i}\right)^{-l_i}
$$
since $\mathcal D_l\left({1\over f_{\Delta^*}(a)}\prod_{i=1}^4
{dX_i\over X_i}\right)=0 $ is now exact equation, it follows from
Theorem \ref{open-close variation} that
\begin{eqnarray}
    \mathcal D_l \Pi(a,b)=0
\end{eqnarray}
There is another set of differential equation given by variation of
the open string modulus. By Theorem \ref{open-close variation} again,
\begin{eqnarray}
    {\pa\over \pa b_0}\Pi(a,b)&=&\int_{\pa{\Gamma_{a,b}}}Res_{Y_{a,b}}Res_{D_b}
 {1\over f_{\Delta^*}(a)h(b)}\prod_{i=1}^4 {dX_i\over  X_i}\label{b0Pi}\\
{\pa\over \pa
b_1}\Pi(a,b)&=&\int_{\pa{\Gamma_{a,b}}}Res_{Y_{a,b}}Res_{D_b}
{\prod_i\left(X^{v_i^*}\right)^{q_i}\over f_{\Delta^*}(a)h(b)}\prod_{i=1}^4 {dX_i\over  X_i}\label{b1Pi}
\end{eqnarray}
It follows that
\begin{eqnarray}\label{homogeneous}
    \left(b_0{\pa\over \pa b_0}+b_1{\pa\over \pa b_1}\right)\Pi(a,b)=0.
\end{eqnarray}
This can also be seen as a consequence of the invariance of the divisor $D_b$ under rescaling $b_0\to
\lambda b_0, b_1\to \lambda b_1$. 

Next, to get a full set of equations governing the relative periods, we introduce the following \cite{Mayr}.
Given the vector $q$ above, the enhanced polytope $\hat\Delta^*$ is the convex hull of the following points in $\mathbb Z^{5}$:
\begin{eqnarray}
    \hat v_i^*=(v_i^*;0), v_i^*\in \Delta^*, \ \ \ \
    w_0=(0;1), \ w_1=(\sum_{v_i^*\in \Delta^*}q_i v_i^*;1).
\end{eqnarray}
Consider their images under the map $\mathbb Z^5\to\mathbb Z^6$, $w\mapsto \bar w=(w,1)$:
\begin{eqnarray}
    \bar{\hat v}_i^*=(1;v_i^*;0), v_i^*\in \Delta^*, \ \ \ \
    \bar w_0=(1;0;1), \ \bar w_1=(1;\sum_{v_i^*\in \Delta^*}q_i v_i^*;1)
\end{eqnarray}
and define the enhanced relation lattice by
\begin{eqnarray}
    \hat L=\left\{\hat l=(\hat l^c; \hat l^o)\in \mathbb Z^{|\Delta^*|+2}|\sum_{v_i^*\in
    \Delta^*}\hat l^c_i \bar{\hat v}_i^* + \sum_{k=0,1}\hat l^o_k \bar w_k=0\right\}.
\end{eqnarray}
(The subscript $"c"$, $"o"$ stand for "closed" and "open"
respectively.) 

\begin{prop} (Enhanced GKZ System)
Put
\begin{eqnarray*}
     \hat{\mathcal L}_k &=&
    \sum_{v_i^*\in \Delta^*}\bar{\hat v}^*_{i,k} a_i{\pa\over \pa a_i}+\bar w_{0,k} b_0{\pa \over \pa b_0}+\bar w_{1,k}b_1{\pa\over \pa b_1}-\hat\beta_k, \ \ \ \ \ k=0,..,5\\
        \hat{\mathcal D}_{\hat l}&=&  \prod_{l_i^c>0}\left({\pa\over \pa
    a_i}\right)^{l^c_i}\prod_{l_k^o>0}\left({\pa\over \pa
   b_k}\right)^{l_k^o}-
   \prod_{l_i^c<0}\left({\pa\over \pa
    a_i}\right)^{-l^c_i}\prod_{l_k^o<0}\left({\pa\over \pa b_k}\right)^{-l_k^o}, \ \ \hat l\in \hat L
\end{eqnarray*}
where $\hat \beta=(-1,0,0,0,0,0)$. Then the relative periods
$\Pi(a,b)$ are solutions to the following enhanced GKZ system:
\begin{eqnarray}
 \hat{\mathcal L}_k \Pi(a,b)&=&0,\ \ \ \  k=0,..,5 \\
  \hat{\mathcal D}_{\hat l}\Pi(a,b)&=&0, \ \ \ \ \hat l\in \hat L \label{D-system}
\end{eqnarray}
\end{prop}

\begin{proof}
The first order equations are simply restating (\ref{Lk}) and (\ref{homogeneous}). 
Note that in (\ref{b0Pi}) and (\ref{b1Pi}) the integrals are over exact, hence closed cycles. Thus there is no contributions from the variations of $a,b$ when we differentiate them with respect to the $a,b$. It follows that when the operator $\hat{\mathcal D}_{\hat l}$ is applied to them, we can interchange the order of the operator and integration. Since the operator kills the integrand, it follows easily that
\begin{eqnarray}
    \hat {\mathcal D}_{\hat l}{\pa\over \pa b_0}\Pi(a,b)=\hat{\mathcal D}_{\hat l}{\pa\over \pa b_1}\Pi(a,b)=0.
\end{eqnarray}
Since $\hat D_{\hat l}\Pi(a,b)$ is homogeneous of degree $-|l_0^o|$ in $b_0, b_1$, this equation implies that
\begin{eqnarray}
    \hat D_{\hat l}\Pi(a,b)=0.
\end{eqnarray}
\end{proof}

We now construct a general solution to the enhanced GKZ system above. The main idea grew out of early attempts to generalize the known examples from mirror symmetry (see for e.g. \cite{Batyrev}.) One of the main problems in closed string mirror symmetry for toric hypersurfaces (and complete intersections) in a complete toric variety was to construct the so-called large radius limit and all analytic solutions to the extended GKZ system  \cite{HKSY}\cite{Hosono-Lian} with $\beta=(-1,0,..,0)$.
It was proved \cite{HLY} that if the projective toric variety $\mathbb P_\Sigma$ is semi-positive, then a large radius limit corresponds exactly to the canonical triangulation associated to the fan of $\mathbb P_\Sigma$, and that the unique powers series solution is given by a Gamma series. To get other analytic solutions, our idea was then to {\it deform} this Gamma series, as in eqn. (3.5) of \cite{Hosono-Lian}. For by the analyticity of Gamma function, the deformed Gamma series is a priori a solution to the GKZ system modulo a set of relations in the deformation parameters $D_0,..,D_p$. Now if all the relations lie in the ring $\mathbb C[D_0,..,D_p]$, then it follows easily that the deformed Gamma series would be a solution modulo the Stanley-Reisner ideal of a maximal triangulation. It was then shown that the semi-positivity condition on $\mathbb P_\Sigma$ guarantees that the relations lie in the appropriate ring. The Stanley-Reisner ring is the cohomology ring of $\mathbb P_\Sigma$, and a complete set of solutions is parameterized by it \cite{Hosono-Lian}. Similar problems for general point configurations in $\mathbb Z^n$ has also been studied \cite{Adolphson}\cite{resonant}.

\ \

We now formulate our general solutions to the enhanced GKZ system. We will show that the results on closed string mirror symmetry discussed above essentially carries over to the open string context with some modifications.  
We begin with a brief review of the general setup. Fix a possibly incomplete simplicial fan $\Sigma$ in $\mathbb Z^n$ such that every maximal cone is $n$ dimensional and that the support $|\Sigma|$ of $\Sigma$ is convex. Let $\Sigma(1)=\{\mu_1,..,\mu_p\}$ be the set of integral generators of the 1-cones, and put $\mu_0=0$ and $\nabla=conv(\Sigma(1)\cup 0)$. 
The Stanley-Reisner ideal of the fan $\Sigma$ is denoted by $SR_\Sigma\subset\mathbb C[D_1,..,D_p]$.  Put $\mathcal A=\{\bar\mu_0,..,\bar\mu_p\}\subset(1,\mathbb Z^n)$ with $\bar w=(1,w)$, and let $L\subset\mathbb Z^{p+1}$ be the lattice of relations of $\mathcal A$. 
We shall {\it assume that the first Chern class $c_1(\mathbb P_\Sigma)$ of the toric variety $\mathbb P_\Sigma$ is semi-positive.} 

\begin{prop}
For each $n$-cone $\sigma$ in $\Sigma$, $\sigma\cap\nabla$ is an $n$-simplex whose vertices are $\mu_0$ and the $n$ generators of $\sigma$. Such simplices $\sigma\cap\nabla$ form a triangulation $\mathcal T_\Sigma$ of $\nabla$.
\end{prop}

\proof
The class $c_1(\mathbb P_\Sigma)$ is represented by  the piecewise linear function $\alpha_\Sigma$ with value $1$ at each $\mu_1,..,\mu_p$.  That $c_1(\mathbb P_\Sigma)$ is semi-positive means that $\alpha_\Sigma$ is convex.
The argument of Theorem 4.10 \cite{HLY} applied to $\Sigma$ shows that each $\mu_i$ must be on the boundary of $\nabla$.
This implies that $\sigma\cap\nabla$ is an $n$-simplex whose vertices are $\mu_0$ and the $n$  generators of $\sigma$.
Since the support $|\Sigma|$ is convex, such $n$-simplices must fill up all of $\nabla$. That they form a triangulation follows easily from that $\Sigma$ is a fan. $\Box$

\ \

The triangulation of $\nabla$ above gives $conv(\mathcal A)=(1,\nabla)$ a triangulation, which we also denote by $\mathcal T_\Sigma$.
 We now consider the GKZ $\mathcal A$-hypergeometric system with $\beta=(-1,0,..,0)$. Let $C(\mathcal T_\Sigma)\subset L^*$ be the cone of $\mathcal T_\Sigma$ convex piecewise linear function modulo linear functions. Consider the following deformed Gamma series (cf.  eqn. (3.5) of \cite{Hosono-Lian}): 
\begin{eqnarray}
B_\Sigma(a)={1\over a_0}\sum_{l\in C(\mathcal T_\Sigma)^\vee}
\frac{\Gamma(-l_0-D_0+1)}{\prod_{j=1}^p\Gamma(l_j+D_j+1)} a^{l+D}.
\end{eqnarray}
For $l\in L$, put $B_l=\frac{\Gamma(-l_0-D_0+1)}{\prod_{j=1}^p\Gamma(l_j+D_j+1)}$, which takes value in ${1\over D_0}\mathbb C[[D_0,..,D_p]]$. Then $\sum_{l\in L} B_la^{l+D}/a_0$ is a constant multiple of the GKZ Gamma series with $(D_0,..,D_p)=(1+\gamma_0,\gamma_1,..,\gamma_p)$ (see section 3 \cite{Hosono-Lian}), and it is killed by the operators $\mathcal D_l$ of the GKZ system \cite{GKZ}. Note that $C(\mathcal T_\Sigma)=\cap_{I\in\mathcal T_\Sigma} K(I)^\vee$ \cite{GKZ}. By the usual product identity of Gamma function, it follows easily that $B_l$ is zero modulo $SR_\Sigma$ for $l\notin C(\mathcal T_\Sigma)^\vee$, and so $B_\Sigma(a)$ modulo $SR_\Sigma$ is killed by the $\mathcal D_l$. Let $J_\Sigma\subset\mathbb C[D_0,..,D_p]$ be ideal generated by $SR_\Sigma$ and the linear forms $\sum_i\bar\mu_i D_i$.

\begin{thm}
Suppose that $c_1(\mathbb P_\Sigma)$ is semi-positive and that $\mathcal T_\Sigma$ is regular. Then the deformed Gamma series $B_\Sigma(a)$ modulo $J_\Sigma$ gives analytic solutions to the $\beta=(-1,0,..,0)$ GKZ $\mathcal A$-hypergeometric system on some domain (i.e. near a large radius limit.) If $\mathbb P_\Sigma$ is nonsingular, then this gives a complete set of solutions.
\end{thm}

\proof
By the argument of Theorem 4.10 \cite{HLY} applied to $\Sigma$ again, we find that every $\mathcal T_\Sigma$-primitive relation $l$ has $l_0\leq 0$. Now, Proposition 4.8 \cite{HLY} shows that $C(\mathcal T_\Sigma)^\vee$ is generated by the $\mathcal T_\Sigma$-primitive relations. 
It follows that $B_l\in\mathbb C[[D_0,..,D_p]]$ for $l\in C(\mathcal T_\Sigma)^\vee$. Thus $B_\Sigma(a)$ modulo $J_\Sigma$ is well-defined. Since $\mathcal T_\Sigma$ is regular, $C(\mathcal T_\Sigma)$ is a maximal cone in the secondary fan of $\mathcal A$. If this cone is not regular, we can subdivide it to obtain a regular maximal cone in it. The dual of this regular cone contains $C(\mathcal T_\Sigma)$ and is generated by an integral basis $l^{(1)},..,l^{(p-n)}$ of $L$. The regular cone corresponds to an affine variety with coordinates $x_k=(-1)^{l_0^{(k)}} a^{l^{(k)}}$, and $B_\Sigma(a)$ becomes
$$
{1\over a_0}\sum_{m=(m_1,..,m_{p-n})\in\mathbb Z^{p-n}_{\geq0}} B_{\sum_k m_k l^{(k)}} x^m.
$$
By the usual product identity of Gamma function, it follows easily that this converges for small $x$.

Since expanding $B_\Sigma(a)$ as a linear combination of monomials in the $D_i$ (say, up to any degree $N$) obviously yields coefficients that are linearly independent, taking $B_\Sigma(a)$ modulo $J_\Sigma$ yields a series of the form
$\sum_i w_i(a)\alpha_i$ where the coefficients $w_i$ are also independent; here the $\alpha_i$ are any homogeneous basis of $\mathbb C[D_0,..,D_p]/J_\Sigma$. In other words, $B_\Sigma(a)$ modulo $J_\Sigma$ yields $dim~\mathbb C[D_0,..,D_p]/J_\Sigma$ independent solutions.  Now, for $\Sigma$ is regular the space of solutions has dimension $vol(\nabla)$, by a theorem of \cite{GKZ}. This is the number of maximal cones in $\Sigma$, which coincides with $dim~\mathbb C[D_0,..,D_p]/J_\Sigma$ by the next theorem.  $\Box$

\begin{rmk}
Given a fan $\Sigma$, it is easy to check in practice the semi-positivity of $c_1(\mathbb P_\Sigma)$ and regularity of $\mathcal T_\Sigma$. Neither condition implies the other in general. The semi-positivity condition holds iff every primitive relation (which is easy to compute) $l$ has $l_0\leq 0$. Since the primitive relations generates $C(\mathcal T_\Sigma)^\vee$, it is also easy to decide if this cone is strongly convex. Regularity of $\mathcal T_\Sigma$ is equivalent to $C(\mathcal T_\Sigma)^\vee$ being strongly convex.
\end{rmk}

Let $\Sigma$ be a possibly incomplete simplicial fan in $N=\mathbb Z^n$, such that every maximal cone is $n$ dimensional, that $|\Sigma|$ is convex, and that there exists a strongly convex continuous piecewise linear function $\psi:|\Sigma|\rightarrow\mathbb R$. The next theorem gives the chow ring of the toric variety $\mathbb P_\Sigma$ and generalizes a well-known result for complete simplicial toric varieties.

\begin{thm}
Under the assumptions above on $\Sigma$, the Chow ring $A_*(\mathbb P_\Sigma,\mathbb Q)$ is isomorphic $\mathbb Q[D_1,..,D_p]/J_\Sigma$, whose dimension is the number of maximal cones in $\Sigma$.
\end{thm}

Our overall strategy follows \cite{Fulton}, but one key step requires a modification that borrows an idea in the proof of \cite{Danilov} which dealt with the case of complete toric varieties. Let's introduce some vocabulary. Let $\{\sigma_i\}_{i\in I}$ be the set of $n$-cones in $\Sigma$. Let $m_i\in N^\vee$ be such that $\psi|\sigma_i=m_i$. By the strong convexity assumption, the $m_i$ are pairwise distinct. So, we can find a point $x_0$ in the interior of $|\Sigma|$, so that the values $m_i(x_0)$ are pairwise distinct. We order the index set $I$ and identify it with
$\{1,2,..,|I|\}$ in this order. We say that a codimension one face of the cone $\sigma_i$ is a {\it shared wall} of $\sigma_i$, if it is of the form $\sigma_k\cap\sigma_i$, and a {\it free wall} if it is not. For each $i\in I$, let $\tau_i$ {\it be the intersection of all free walls of $\sigma_i$ and all shared walls $\sigma_k\cap\sigma_i$ with $k>i$.} The following lemma generalizes the key property (*) in section 5.2 \cite{Fulton}.

\begin{lem}(Order Lemma) If $\tau_i\subset\sigma_j$ then $j\geq i$.
\end{lem}
\proof
We begin with some basic facts. Let $i,k\in I$.

\item{(1)} $m_k-m_i\in\sigma_i^\vee$. In particular, $x_0\notin\sigma_i$ for $i>1$. This follows from the convexity of $\psi$ and that $m_1(x_0)<m_2(x_0)<\cdots$.

\item{(2)} Under the inclusion reversing correspondence between the faces of $\sigma_i$ and those of $\sigma_i^\vee$, a wall of $\sigma_i$ corresponds to an edge of $\sigma_i^\vee$ of the form $\mathbb R_+ m$ where $m\in N^\vee$ is an inward pointing normal to the wall. Note that we can pick $m=m_k-m_i$, if the wall is $\sigma_k\cap\sigma_i$.  Note also that $x\in\sigma_i$ iff $m(x)\geq0$ for every edge $\mathbb R_+m$ of $\sigma_i^\vee$.




\ \

Now given $\tau_i\subset\sigma_j$,  for some $j$. We consider the fan $star(\tau_i)$ in $N/\mathbb R \tau_i$, and note that $\psi-m_i$ induces a convex 
piecewise linear function on $star(\tau_i)$. We denote by $\bar \sigma_i, \bar \sigma_j$ the images of $\sigma_i,\sigma_j$ in $star(\tau_i)$, and 
$\bar x_0$ the image of $x_0$ in $N/\mathbb R \tau_i$. Then $\bar\sigma_i^\vee=\sigma_i^\vee\cap \tau_i^\bot$. Let $m$ be an edge of 
$\bar\sigma_i^\vee$. Then $m$ corresponds to free wall in $\sigma_i$ or a shared wall  $\sigma_k\cap \sigma_i$ with $k>i$. In the first case, we have
$m(x_0)\geq 0$ since $x_0$ lies in $|\Sigma|$ which is convex by assumption, and in the second case, $m(x_0)=m_k(x_0)-m_i(x_0)>0$ since $k>i$. It follows that 
$\bar x_0\in \bar \sigma_i$. Therefore $(m_j-m_i)(\bar x_0)\geq 0$ by the convexity of $\psi-m_i$, i.e., $m_j(x_0)\geq m_i(x_0)$. It follows that  $j\geq i$. $\Box$

\ \

\noindent{\it Proof of Theorem:}
\ \

\noindent \bl{Claim}: For each cone $\gamma$ in $\Sigma$, there's a unique $i=i(\gamma)$ such that $\tau_i\subset \gamma\subset \sigma_i$. And if $\gamma\subset \gamma^\prime$, then $i(\gamma)\leq i(\gamma^\prime)$. \\
Uniqueness follows from the Order  Lemma above. For the existence, let $i(\gamma)$ be the minimal such that $\gamma\subset \sigma_i$. If $\gamma=\sigma_i$, then we're done, otherwise, we write $\gamma$ as intersection of (n-1)-dim faces. Then it's easy to see that by the minimal property of $i(\gamma)$, we have $\tau_i\subset \gamma$. The claim is proved. \\
\\
It follows from the argument in section 5.2 \cite{Fulton} that 
$$
	A_*(P_{\Sigma})=H_{*}^{BM}(P_{\Sigma})=span_{\mathbb Q}\{[V(\tau_i)]_i\}
$$
where $P_\Sigma$ is the toric variety associated to $\Sigma$. Now we consider the surjection 
$$
	A/J_\Sigma\to A_*(P_{\Sigma}).
$$
where $A=\mathbb Q[D_1,..,D_p]$.
The ``algebraic moving lemma'' in section 5.2 \cite{Fulton} continues to work in this the incomplete case, and the proof there shows that $\{p(\tau_i)\}$ generates $A/J_\Sigma$ as $\mathbb Q$-module, where $p(\tau_i)$ is the monomial corresponds to the cone $\tau_i$. By comparing the dimensions, we get the isomorphism 
$$
A/J_\Sigma\simeq A_*(P_{\Sigma})
$$
In particular, $\dim_{\mathbb Q}A/J_\Sigma$ equals the number of maximal cones in $\Sigma$. $\Box$





\ \

We now apply our results to solve the enhanced GKZ system in open string mirror symmetry as a special case.
Let $\Sigma^*$ be a complete regular fan in $\mathbb R^4$, $\Sigma^*(1)$ be the set of generators of its 1-dimensional cones, and assume that $\Sigma^*(1)=\Delta^*$. Let $\hat\Delta^*$ be the enhanced polytope as before. For each cone $\sigma\in\Sigma^*$, we get a cone $\hat\sigma$ in $\mathbb R^5$ generated by $(\sigma,0)$ and the vector $w_1$. The set of cones $\hat\sigma$ obtained this way, together with 0, form a fan whose support is the half space $\mathbb R^4\times\mathbb R_\geq$. Since $w_0$ lies in this half space, there is a canonical way to subdivide the fan into a regular incomplete fan $\hat\Sigma$, so that $\hat\Sigma(1)=\hat\Delta^*$. We shall {\it assume that  $c_1(\mathbb P_{\hat\Sigma})$ is semi-positive and that $\mathcal T_\Sigma$ is regular.}  Note that this implies that $c_1(\mathbb P_{\Sigma^*})$ is also semi-positive. But the converse is not true. In fact, semi-positivity of $c_1(\mathbb P_{\hat\Sigma})$  put a strong constraint on the B-brane vector $q$. 

\begin{cor}
Suppose that $c_1(\mathbb P_{\hat\Sigma})$ is semi-positive and $\mathcal T_{\hat\Sigma}$ is regular. Then $B_{\hat\Sigma}(a)$ modulo $J_\Sigma$ gives a complete set of solutions to the enhanced GKZ system.
\end{cor}

\subsection*{Example: B-brane on Mirror Quintic}
Consider the one-parameter family of mirror quintic given in homogeneous coordinate by
\begin{eqnarray}
    \prod_{i=1}^5 z_i-x^{1/5}\sum_{i=1}^5 z_i^5=0
\end{eqnarray}
and one-parameter family of B-branes on it given by
\begin{eqnarray}
    z_5^4-\phi z_1z_2z_3z_4=0
\end{eqnarray}
Note that this divisor is invariant under the $(\mathbb Z_5)^3$
action \cite{Candelas}, hence descends to a hypersurface on the
mirror quintic. The integral points of the polytope $\Delta^*$ are
\begin{eqnarray*}
 \Delta^*: && v^*_0=(0,0,0,0)\\&& v^*_1=(1,0,0,0)\\ && v^*_2=(0,1,0,0)\\ &&
              v^*_3=(0,0,1,0)\\ && v^*_4=(0,0,0,1)\\ && v^*_5=(-1,-1,-1,-1).
\end{eqnarray*}
The toric coordinates $X_i$ are related to homogeneous coordinates $z_i$ by
\begin{eqnarray}
    X_i={z_i^5\over \prod_{i=1}^5 z_i},\ \ i=1,...,4.
\end{eqnarray}
Hence the B-brane divisor is given in toric coordinates by
$$
    X^{v_5^*-v_0^*}-\phi=0,
$$
which corresponds to the B-brane vector
$$
    q=(-1,0,0,0,0,1).
$$
The integral points in the enhanced polytope  are
\begin{eqnarray*}
 \hat \Delta^*: && \hat v^*_0=(0,0,0,0;0)\\
 && \hat v^*_1=(1,0,0,0;0)\\ && \hat v^*_2=(0,1,0,0;0)\\ &&
              \hat v^*_3=(0,0,1,0;0)\\ && \hat v^*_4=(0,0,0,1;0)\\ && \hat v^*_5=(-1,-1,-1,-1;0)\\
              && w_0=(0,0,0,0;1)\\
              && w_1=(-1,-1,-1,-1;1)
\end{eqnarray*}
The relation lattice is generated by
\begin{eqnarray}
    l^{(0)}&=&(-1,0,0,0,0,1,1,-1)\\
    l^{(1)}&=&(-5,1,1,1,1,1,0,0)
\end{eqnarray}
$x$ above is one of the moduli variable, and the relation between $\phi$ and $u$ can be read
$$
    u=x^{1/5} \phi
$$
The Picard-Fuchs equation from $l^{(0)}, l^{(1)}$ can be read
\begin{eqnarray*}
\left\{ \left(\theta_x+\theta_u\right)\theta_u-u\left((5\theta_x+\theta_u+1)\right)\theta_u\right\} \tilde\Pi(x,u)=0\\
\left\{\theta_x^4 (\theta_x+\theta_u)-x\prod_{i=1}^5(5\theta_x+\theta_u+i)\right\} \tilde \Pi(x,u)=0
\end{eqnarray*}
where $\theta_x=x{\pa\over \pa x}, \theta_u=u{\pa\over \pa u}$. This is equivalent to the ones in \cite{Mayr}. To get the complete equations from 
GKZ system, we need one more equation obtained from $l^{(1)}-l^{(0)}$ \cite{Mayr}. 

\subsubsection*{Large Radius Limit}
The following maximal triangulation $\mathcal T=\mathcal T_{\hat\Sigma}$ of $\hat \Delta^*$
corresponds to the large radius limit. Its maximal simplices correspond to the following maximal cones in $\hat\Sigma$
\begin{eqnarray*}
          <\hat v_0^*,\hat v_1^*,\hat v_2^*, \hat v_3^*, v_4^*, w_0>\\
          <\hat v_0^*,\hat v_1^*,\hat v_2^*, \hat v_3^*, v_5^*, w_1>\\
          <\hat v_0^*,\hat v_1^*,\hat v_2^*, \hat v_4^*, v_5^*, w_1>\\
          <\hat v_0^*,\hat v_1^*,\hat v_3^*, \hat v_4^*, v_5^*, w_1>\\
          <\hat v_0^*,\hat v_2^*,\hat v_3^*, \hat v_4^*, v_5^*, w_1>\\
          <\hat v_0^*,\hat v_1^*,\hat v_2^*, \hat v_3^*, w_0, w_1>\\
          <\hat v_0^*,\hat v_1^*,\hat v_2^*, \hat v_4^*, w_0, w_1>\\
          <\hat v_0^*,\hat v_1^*,\hat v_3^*, \hat v_4^*, w_0, w_1>\\
          <\hat v_0^*,\hat v_2^*,\hat v_3^*, \hat v_4^*, w_0, w_1>
\end{eqnarray*}
where generators of each cone is given in the bracket $<\cdots>$.
It follows  that the cone $C(\mathcal T)$ in the secondary fan
corresponding to $\mathcal T$ \cite{GKZ} \cite{Hosono-Lian} has the
dual cone
$$
    C(\mathcal T)^\vee=\mathbb Z_{\geq 0} (l^{(1)}-l^{(0)})+\mathbb Z_{\geq 0} l^{(0)}
$$
From this, it is easy to see that $c_1(\mathbb P_{\hat\Sigma})$ is semi-positive and that $ \mathcal T$ is regular.
Local coordinates are given by
$$
    z_1={a_1a_2a_3a_4b_1\over a_0^4 b_0}, \ \ \ z_2={a_5b_0\over a_0b_1}
$$
where $z_1=0, z_2=0$ is the large radius limit point.  We use
variables $D_0, D_1, \ldots, D_5$ to represent $\hat v_0^*,\hat v_1^*,\ldots, \hat v_5^*$ and $D_6, D_7$ to represent 
$w_0, w_1$. The primitive collections
gives the generators of the Stanley-Reisner ideal
$$
    D_1D_2D_3D_4D_7, D_5D_6, D_1D_2D_3D_4D_5
$$
There're only two independent $D_i$'s after imposing the linear relations as in the above discussion on deformed Gamma series. Put
$$
	E_1=D_1, E_2=D_5.
$$
Then the deformed Gamma series can be written as 
$$
	B_{\hat \Delta^*}(z_1,z_2)={1\over a_0}\sum_{m,n\geq 0}{ (-1)^m \Gamma(4m+n+4E_1+E_2+1)\sin{\pi(E_1-E_2)}\over
	\Gamma(m+E_1+1)^4\Gamma(n+E_2+1) \pi(m-n+E_1-E_2)
	} z_1^{m+E_1} z_2^{n+E_2}
$$
where we have used Gamma function identity: $\Gamma(1+z)\Gamma(1-z)={\pi z\over \sin{\pi z}}$. The generators of Stanley-Reisner ideal
can be written modulo the linear relations as $E_1^4(E_1-E_2), E_2(E_1-E_2), E_1^4 E_2$. Then $B_{\hat \Sigma}(z_1,z_2)$ satisfies the 
GKZ equations if it's viewed as taking values in the ring 
$$
	\mathbb{Q}[E_1, E_2]/(E_1^4(E_1-E_2), E_2(E_1-E_2), E_1^4 E_2)
$$
which has length $9$ equal to $vol(conv((1, \hat\Delta^*)))$. Therefore we have obtained the 9-dimensional solution space to the enhanced GKZ system.
 There is one regular
solution which coincides with the regular closed string period for the mirror
quintic at the large radius limit:
$$
    \omega_0(z_1,z_2)=\sum_{m\geq 0}{(-1)^m\ (5m)!\over (m!)^5}(z_1z_2)^m
$$
and two solutions with single log behavior
\begin{eqnarray*}
    \omega_1(z_1,z_2)&=&\omega_0(z_1,z_2) \ln z_1z_2 +5 \sum_{m\geq 0}{(-1)^m\ (5m)! \sum\limits_{j=m+1}^{5m} {1\over j}\over (m!)^5}(z_1z_2)^m\\
    \omega_2(z_1,z_2)&=&\omega_0(z_1,z_2) \ln z_2+\sum_{m\geq 0}{(-1)^m\ (5m)! \sum\limits_{j=m+1}^{5m} {1\over j}\over (m!)^5}(z_1z_2)^m
    -\sum_{m\geq 0, n\geq 0, m\neq n} {(-1)^{m}(4m+n)!\over (m!)^4 n! (m-n)}z_1^m z_2^n
\end{eqnarray*}
$\omega_1(z_1,z_2)$ corresponds to the closed period, and $\omega_2(z_1,z_2)$ now comes from the relative period. The so-called open-closed mirror map
can be similarly obtained as in the closed string case by normalizing
$$
    t_1=\omega_1(z_1,z_2)/\omega_0(z_1,z_2), t_2=\omega_2(z_1,z_2)/\omega_0(z_1,z_2)
$$

Not all of our solutions to the enhanced GKZ system come from relative periods. In closed string mirror symmetry, the periods of toric Calabi-Yau hypersurfaces in $\mathbb P^{(4)}_\Sigma$ near the large radius limit correspond to solutions with no more than $(log)^3$ behavior, while the full solution space to the corresponding GKZ system include functions with up to $(log)^4$. A key observation in \cite{Hosono-Lian} was that a natural way to get the periods from the deformed Gamma series is by multiplying it by the Calabi-Yau divisor $c_1(\mathbb P_\Sigma)$. This has the effect of killing off the $(log)^4$ terms in $B_\Sigma$. We expect that the same phenomenon happens for relative periods. Namely, they should correspond exactly to the solution $c_1(\mathbb P_{\hat\Sigma})B_{\hat\Sigma}$. Again, this does have the expected effect of killing off the $(log)^5$ terms in $B_{\hat\Sigma}$. 
In the example above this procedure yields 7 independent solutions as in \cite{Mayr} \cite{Masoud}. This will be studied in greater generality in a follow up paper.

\subsection*{Relative Periods And Abel-Jacobi Map}
Consider the relative period $\Pi(z,u)$ given by the family of pairs
$(X_z, Y_{z,u})$ as before. From the double residue formula
\begin{eqnarray*}
 \pa_u \Pi(z,u)&=&\int_{\pa{\Gamma_{z,u}}}Res_{Y_{z,u}}Res_{D_u}(\pa_u \log
  Q_u)
    \omega_z
\end{eqnarray*}
we see that if $u_0$ is some point where $\pa
{\Gamma_{z,u_0}}=C^+-C-$ is a pair of holomorphic curves, then
$$
    \pa_{u}\Pi(z,u)|_{z,u_0}=0
$$
since
the integrand is a form of type $(2,0)$.
Therefore the loci corresponding to Abel-Jacobi map lies in the
critical loci of relative period with respect to the deformation of
the divisor. This corresponds in physics the statement that the
critical point of the off-shell superpotential will give the D-brane
domain wall tension. This was carried out in details for the mirror quintic
example in \cite{Masoud} and \cite{Mayr}. Note that in their
example, the pair of curves $C^+, C^-$ lie inside the
transversal-intersection loci of $X_z\cap D_u$ except at two fixed
points. The double residue formula doesn't work in a
straight-forward way. But it can be seen that by blowing up twice at
those two points, the proper transformation of $X_z, Q_u$ will
intersect transversally at all points on the proper transformation
of $C^{\pm}$. The double residue formula can then be applied to the blown up configuration. The
Abel-Jacobi map will still be the critical value of some of the
relative periods.


  \appendix


S. Li, Department of Mathematics, Harvard University, Cambridge MA 02138. 

B. Lian, Department of Mathematics, Brandeis University, Waltham MA 02454. 

S.-T. Yau, Department of Mathematics, Harvard University, Cambridge MA 02138.

\end{document}